\documentclass[a4paper,11pt]{article}

\usepackage[english]{babel}
\usepackage{cite}
\usepackage{amsfonts}
\usepackage{latexsym}
\usepackage{amsthm}
\usepackage{amsopn}
\usepackage{amsmath}
\usepackage[all]{xy}
\usepackage{verbatim}
\usepackage{amssymb}
\usepackage{mathrsfs}
\usepackage{float}
\usepackage[active]{srcltx}
\usepackage{fullpage}
\usepackage{hyperref}
\usepackage{manyfoot}
\usepackage{amscd}
\usepackage[dvips]{graphics}
\usepackage[dvips]{graphicx}
\usepackage{amscd}
\usepackage{color}
\usepackage{cite}

\newtheorem{inizio}{Lemma}[section]
\newtheorem{theorem}[inizio]{Theorem}
\newtheorem{corollary}[inizio]{Corollary}
\newtheorem{proposition}[inizio]{Proposition}
\newtheorem{lemma}[inizio]{Lemma}

\newtheorem{definition}[inizio]{Definition}
\newtheorem{remark}[inizio]{Remark}

\newtheorem*{theoA}{Main Theorem}

\newcommand{\oo}{\mathscr{O}}

\newcommand{\Hom}{{\mathrm{Hom}}}

\newcommand{\ZZ}{\mathbb{Z}}

\newcommand{\CC}{\mathbb{C}}

\newcommand{\EE}{\mathscr{E}} 
\newcommand{\FF}{\mathscr{F}} 
\newcommand{\LL}{\mathscr{L}} 

\newcommand{\wA}{\widehat{A}}
\newcommand{\wS}{\widehat{S}}
\newcommand{\wR}{\widehat{R}}

\newcommand{\lr}{\longrightarrow}

\renewcommand{\to}{\longrightarrow}

\title{A new family of  surfaces with $p_g=q=2$ and $K^2=6$ whose Albanese map has degree $4$}

\author{Matteo Penegini and Francesco Polizzi}

\date{}

\begin{document}

\maketitle

\begin{abstract}
We construct a new family of minimal surfaces of general type with
$p_g=q=2$ and $K^2=6$, whose Albanese map is a quadruple cover of an
abelian surface with polarization of type $(1, \,3)$. We also show
that this family provides an irreducible component of the moduli
space of surfaces with $p_g=q=2$ and $K^2=6$. Finally, we prove that
such a component is generically smooth of dimension $4$ and that it
contains the $2$-dimensional family of product-quotient examples
previously constructed by the first author. The main tools we use
are the Fourier-Mukai transform and the Schr\"odinger representation
of the finite Heisenberg group $\mathscr{H}_3$.
\end{abstract}

\Footnotetext{{}}{\textit{2010 Mathematics Subject Classification}:
14J29, 14J10}

\Footnotetext{{}} {\textit{Keywords}: Surface of general type,
Albanese map, Fourier-Mukai transform}



\section{Introduction}\label{sec.intro}

In recent years, both the geographical problem and the fine
classification for \emph{irregular} algebraic surfaces (i.e.
surfaces with irregularity $q>0$) have attracted the attention of
several authors; in particular minimal surfaces of general type with
$\chi(\oo)=1$, i.e. $p_g=q$, were deeply investigated.

In this case some well-known results imply $p_g \leq 4$ . While
surfaces with $p_g=q=4$ and $p_g=q=3$ have been completely
described, see \cite{D82}, \cite{CCML98}, \cite{HP02}, \cite{Pi02},
the classification of those with $p_g=q=2$ is still missing (see
\cite{PP10} and \cite{PP11} for a recent account on this topic). As
the title suggests, in the present paper we consider some new
surfaces $\widehat{S}$ with $p_g=q=2$ and $K^2=6$ whose Albanese map
$\hat{\alpha} \colon \wS \to \wA:=\textrm{Alb}(\wS)$ is a quadruple
cover of an abelian surface $\wA$.

Our construction presents some analogies with the one
 presented in \cite{CH06} and \cite{PP10} for the case $p_g=q=2$ and $K^2=5$.
 Indeed, in both situations the
Tschirnhausen bundle
 $\mathscr{E}^{\vee}$ associated with the Albanese cover
 is of the form
$\Phi^{\mathscr{P}}(\mathscr{L}^{-1})$, where $\mathscr{L}$ is a
polarization on $A$ (the dual abelian variety of $\wA$) and
$\Phi^{\mathscr{P}}$ denotes the Fourier-Mukai transform. More
precisely, in the case studied in \cite{CH06} the surfaces are
 triple covers, $\mathscr{E}^{\vee}$ has rank $2$ and $\mathscr{L}$ is a polarization of type $(1, \,
2)$; in our setting, instead, the cover has degree $4$, the bundle
$\mathscr{E}^{\vee}$ has rank $3$ and $\mathscr{L}$ is a
polarization of type $(1, \, 3)$.

The results of the paper can be summarized as follows, see Theorem
\ref{theo.main}.

\begin{theoA}
There exists a $4$-dimensional family $\mathcal{M}_{\Phi}$ of
minimal surfaces of general type with $p_g=q=2$ and $K^2=6$ such
that, for the general element $\widehat{S} \in \mathcal{M}_{\Phi}$,
the canonical class $K_{\wS}$ is ample and the Albanese map
$\hat{\alpha}\colon \widehat{S} \longrightarrow \widehat{A}$ is a
finite cover of degree $4$.

The Tschirnhausen bundle $\EE^{\vee}$ associated with
$\hat{\alpha}$ is isomorphic to
$\Phi^{\mathscr{P}}(\mathscr{L}^{-1})$, where $\mathscr{L}$ is a
polarization of type $(1, \, 3)$ on $A$.

The family $\mathcal{M}_{\Phi}$ provides an irreducible component of
the moduli space $\mathcal{M}_{2, \, 2, \,6}^{\emph{can}}$ of
canonical models of minimal surfaces of general type with $p_g=q=2$
and $K^2=6$. Such a component is generically smooth and contains the
$2$-dimensional family of product-quotient surfaces constructed by
the first author in \emph{\cite{Pe11}}.
\end{theoA}

The Main Theorem is obtained by extending the construction given in
\cite{CH06} to the much more complicated case of quadruple covers.
More precisely, in order to build a quadruple cover $\hat{\alpha}
\colon \widehat{S} \lr \widehat{A}$ with Tschirnhausen bundle
$\EE^{\vee}$, we first build a quadruple cover $\alpha \colon S \lr
A$ with Tschirnhausen bundle $\phi_{\LL^{-1}}^*\EE^{\vee}$ (here
$\phi_{\LL} \colon A \to \widehat{A}$ denotes the group homomorphism
sending $x \in A$ to $t_x^* \mathscr{L}^{-1} \otimes \mathscr{L} \in
\widehat{A}$) and then, by using the Schr\"odinger representation of
the finite Heisenberg group $\mathscr{H}_3$ on $H^0(A, \, \LL)$, we
identify those covers of this type that descend to a quadruple cover
$\hat{\alpha} \colon \widehat{S} \lr \widehat{A}$. For the general
surface $\wS$, the branch divisor $\widehat{B} \subset \wA$ of
$\alpha \colon \wS \lr \wA$ is a curve in the linear system
$|\widehat{\mathscr{L}}^{\otimes 2}|$, where $\widehat{\mathscr{L}}$
is a polarization of type $(1, \,3)$ on $\wA$, with six ordinary
cusps and no other singularities; such a curve is
$\mathscr{H}_3$-equivariant and can be associated with the dual of a
member of the Hesse pencil of plane cubics in $\mathbb{P}^2$.

Let us explain now the way in which this paper is organized.

In Section \ref{sec.prel} we set up notation and terminology and we
present some preliminaries. In particular we recall the theory of
quadruple covers developed by Casnati-Ekedahl \cite{CE96} and
Hahn-Miranda \cite{HM99} and we briefly describe the geometry of the
Hesse pencil.

Sections \ref{sec.constr}, \ref{sec.moduli}, \ref{sec.simple} are
devoted to the proof of the Main Theorem.

In Section \ref{sec.constr} we present our construction, we prove
that a general surface $\widehat{S}$ in our family is smooth and we
compute its invariants (Propositions \ref{prop.sing.covers} and
\ref{prop.num.invariants.cover}).

In Section \ref{sec.moduli} we examine the subset of the moduli
space corresponding to our surfaces, showing that it is generically
 smooth, of dimension $4$ and that its closure $\mathcal{M}_{\Phi}$ provides an
 irreducible component of $\mathcal{M}_{2, \, 2, \,
 6}^{\textrm{can}}$ (Proposition \ref{prop.irred.comp}).
 We also observe that the general surface in
 $\mathcal{M}_{\Phi}$ admits no pencil over a curve of strictly
 positive genus (Proposition \ref{rem-no-irr}).

In Section \ref{sec.simple} we prove that the moduli space
$\mathcal{M}_{2, \, 2, \,
 6}^{\textrm{can}}$ contains
a $3$-dimensional singular locus (Corollary \ref{cor.sing.loc.M}).
 Moreover, we show that the irreducible component $\mathcal{M}_{\Phi}$ contains the
$2$-dimensional family of product-quotient examples constructed by
the first author in \cite{Pe11} (Proposition \ref{wT-in-Mphi}).

Finally, in Section \ref{sec.open.prob} we present some open
problems.

The paper also contains two appendices. In the Appendix $1$ we
show the following technical result needed in the proof of the
Main Theorem: for a general choice of the pair $(A, \,
\mathscr{L})$, the three distinguished divisors coming from the
Schr\"odinger representation of the Heisenberg group
$\mathscr{H}_3$ on $H^0(A, \, \mathscr{L})$ are smooth and
intersect transversally. Appendix $2$ contains the computer
algebra script we used to compute the equation of the branch curve
$B$ of $\alpha \colon S \to A$; it is written in the Computer
Algebra System \verb|MAGMA|, see \cite{Magma}.


\bigskip
\textbf{Acknowledgments} Both authors were partially supported by
the DFG Forschergruppe 790 \emph{Classification of algebraic
surfaces and compact complex manifolds}.

F. Polizzi was partially supported by  Progetto MIUR di Rilevante
Interesse Nazionale \emph{Geometria delle Variet$\grave{a}$
Algebriche e loro Spazi di Moduli} and by GNSAGA-INdAM. He thanks
the Mathematisches Institut, Universit$\ddot{\textrm{a}}$t Bayreuth
for the invitation and the hospitality in the period
October-November 2011.

The authors are indebted with F. Catanese for sharing with them some
of his ideas on this subject. They are also grateful to I. Bauer, W.
Liu,  E. Sernesi, B. van Geemen for useful discussions and
suggestions. Finally they thank the referee, whose comments
considerably improved the presentation of these results.

\bigskip

\textbf{Notation and conventions.} We work over the field
$\mathbb{C}$ of complex numbers.

If $A$ is an abelian variety, we call $\wA:= \textrm{Pic}^0(A)$ its
dual abelian variety.

If $\mathscr{L}$ is a line bundle on $A$ we denote by
$\phi_{\mathscr{L}}$ the morphism $\phi_{\LL} \colon A
\longrightarrow \wA$ given by $x \mapsto t^*_x \LL \otimes
\LL^{-1}$. If $\LL$ is non-degenerate then $\phi_{\LL}$ is an
isogeny, whose kernel is denoted by $K(\LL)$. In this case the
\emph{index} of $\LL$ is the unique integer $i(\LL)$ such that
$H^j(A, \, \LL)=0$ unless $j=i(\LL)$.

Throughout the paper we use italic letters for line bundles  and
 capital letters for the corresponding Cartier divisors, so we write for
instance $\mathscr{L}=\oo_S(L)$. The corresponding complete linear
system is denoted either by $|\mathscr{L}|$ or by $|L|$.

If $L$ is an ample divisor on $A$, then it defines a positive line
bundle $\mathscr{L}=\oo_A(L)$, whose first Chern class is a
polarization on $A$. By abuse of notation we consider both the line
bundle $\mathscr{L}$ and the divisor $L$  as polarizations.

By \emph{surface} we mean a projective, non-singular surface $S$,
and for such a surface $\omega_S=\oo_S(K_S)$ denotes the canonical
class, $p_g(S)=h^0(S, \, \omega_S)$ is the \emph{geometric genus},
$q(S)=h^1(S, \, \omega_S)$ is the \emph{irregularity} and
$\chi(\mathscr{O}_S)=1-q(S)+p_g(S)$ is the \emph{Euler-Poincar\'e
characteristic}.

\section{Preliminaries} \label{sec.prel}
\subsection{Quadruple cover of algebraic varieties}
The two papers \cite{CE96} and  \cite{HM99} deal with the quadruple
covers of algebraic varieties; the former only considers the
Gorenstein case, whereas the latter develops the theory in full
generality. The main results are the following.

\begin{theorem}\emph{\cite[Theorems 1.6 and 4.4]{CE96}}\label{theo.casnati}
Let $Y$ be a smooth algebraic variety. Any quadruple Gorenstein
cover $f\colon X \longrightarrow Y$ is determined by a locally free
$\mathscr{O}_{Y}-$module $\EE$ of rank $3$, a locally free
$\mathscr{O}_{Y}-$module $\FF$ of rank $2$ with $\bigwedge^2 \FF
\cong \bigwedge^3 \EE^{\vee}$ and a general section $\eta \in H^0(Y,
\,  S^2\EE^{\vee} \otimes \FF^{\vee} )$.
\end{theorem}

\begin{theorem}\emph{\cite[Theorem 1.2]{HM99}}\label{theo.mir}
Let $Y$ be a smooth algebraic variety. Any quadruple cover $f\colon
X \longrightarrow Y$ is determined by a locally free
$\mathscr{O}_Y$-module $\mathscr{E}$ of rank $3$ and a totally
decomposable section $\eta \in H^0(Y, \, \bigwedge^2 S^2
\mathscr{E}^{\vee} \otimes \bigwedge^3 \mathscr{E})$.
\end{theorem}

In order to make the notation consistent, in Theorem
\ref{theo.casnati} we called $\EE^{\vee}$ the sheaf which is called
$\EE$ in \cite{CE96}. In Theorem \ref{theo.mir}, \emph{totally
decomposable} means that, for all $y \in Y$, the image of
\begin{equation*}
\eta|_y \colon (\bigwedge^3 \EE^{\vee})_y \longrightarrow
(\bigwedge^2 S^2 \mathscr{E}^{\vee})_y
\end{equation*}
is totally decomposable in $(\bigwedge^2 S^2 \mathscr{E}^{\vee})_y$,
i.e. of the form $\xi_1 \wedge \xi_2$ with $\xi_i \in (S^2
\mathscr{E}^{\vee})_y$.

The vector bundle $\EE^{\vee}$ is called the \emph{Tschirnhausen
bundle} of the cover. We have $f_{*}\mathscr{O}_X= \mathscr{O}_Y
\oplus \EE$, so
\begin{equation}\label{eq.inv.hacca}
h^i(X, \, \mathscr{O}_X)=h^i(Y, \, \mathscr{O}_Y)+h^i(Y, \, \EE).
\end{equation}

By \cite[Proposition 5.1]{CE96} there is an exact sequence
\begin{equation}\label{suc.effe}
0 \longrightarrow \FF \stackrel{\varphi}{\longrightarrow} S^2
\EE^{\vee} \longrightarrow f_*\omega^{2}_{X | Y} \longrightarrow 0
\end{equation}
and the associated Eagon-Northcott complex tensored with
$\bigwedge^2 \FF^{\vee}$ yields
\begin{equation}\label{suc.igor.notcot}
0 \longrightarrow S^2\FF \otimes \bigwedge^2 \FF^{\vee}
\longrightarrow S^2 \EE^{\vee} \otimes \FF^{\vee} \longrightarrow
\bigwedge^2 S^2 \EE^{\vee} \otimes \bigwedge^3 \EE.
\end{equation}

The induced map in cohomology
\begin{equation*}
H^0(Y, \, S^2 \EE^{\vee} \otimes \FF^{\vee}) \longrightarrow H^0(Y,
\, \bigwedge^2 S^2 \EE^{\vee} \otimes \bigwedge^3 \EE)
\end{equation*}
provides the bridge between Theorem \ref{theo.casnati} and Theorem
\ref{theo.mir}. In fact a straightforward computation shows that it
sends the element
 $\varphi \in \textrm{Hom}(\FF, \,  S^2 \EE^{\vee}) \cong H^0(Y,
\, S^2 \EE^{\vee} \otimes \FF^{\vee})$ to the totally decomposable
element $2 \varphi \wedge \varphi \in \textrm{Hom}(\bigwedge^2 \FF,
\, \bigwedge^2S^2 \EE^{\vee}) \cong H^0(Y, \, \bigwedge^2 S^2
\EE^{\vee} \otimes \bigwedge^3 \EE)$.

When $X$ and $Y$ are smooth surfaces one has the following useful
formulae.
\begin{proposition}\emph{\cite[Proposition 5.3]{CE96}}\label{prop.invariants}
Let $X$ and $Y$ be smooth, connected, projective surfaces, $f\colon
X \longrightarrow Y$ a cover of degree $4$ and $R \subset X$ the
ramification divisor of $f$. Then:
\begin{itemize}
\item[$\boldsymbol{(i)}$] $\chi (\mathscr{O}_X)=
4\chi(\mathscr{O}_Y) +
\frac{1}{2}c_1({\EE^{\vee}})K_Y+\frac{1}{2}c^2_1(\EE^{\vee})-c_2(\EE^{\vee})$;
\item[$\boldsymbol{(ii)}$] $K^2_X= 4K^2_Y +
4c_1({\EE^{\vee}})K_Y+2c^2_1(\EE^{\vee})-4c_2(\EE^{\vee})+c_2(\FF)$;
\item[$\boldsymbol{(iii)}$] $p_a(R)= 1 +
c_1({\EE^{\vee}})K_Y+2c^2_1(\EE^{\vee})-4c_2(\EE^{\vee})+c_2(\FF)$.
\end{itemize}
\end{proposition}

\subsection{Fourier-Mukai transforms of (W)IT-sheaves}\label{sub.fourier}
Let $A$ be an abelian variety of dimension $g$ and $\wA:=
\textrm{Pic}^0(A)$ its dual abelian variety. We say that a coherent
sheaf $\mathscr{F}$ on $A$ is
\begin{itemize}
\item a \emph{IT-sheaf of index i} (or, equivalently, that ${\FF}$ \emph{satisfies IT of index i})
if
\begin{equation*}
H^j(A, \, \mathscr{F} \otimes \mathscr{Q})=0 \quad \textrm{for all }
\mathscr{Q} \in \textrm{Pic}^0(A) \quad \textrm{and } j\neq i;
\end{equation*}
\item a \emph{WIT-sheaf of index i} (or, equivalently, that ${\FF}$ \emph{satisfies WIT of index i})
if \begin{equation*} R^j \pi_{\wA \, *}(\mathscr{P} \otimes
\pi^*_{A}\mathscr{F})=0 \quad \textrm{for all} \quad j \neq i,
\end{equation*}
where $\mathscr{P}$ is the normalized Poincar\'e bundle on $A \times
\wA$.
\end{itemize}

If $\FF$ is a WIT-sheaf of index $i$, the coherent sheaf
\begin{equation*}
\Phi^{\mathscr{P}}(\mathscr{F}):={R}^i\pi_{\wA \, *}(\mathscr{P}
\otimes \pi^*_{A}\mathscr{F})
\end{equation*}
is called the \emph{Fourier-Mukai transform} of $\mathscr{F}$.

For simplicity of notation, given any WIT-sheaf $\mathscr{G}$ of
index $i$ on $\widehat{A}$ we use the same symbol
$\Phi^{\mathscr{P}}$ for its Fourier-Mukai transform
\begin{equation*}
\Phi^{\mathscr{P}}(\mathscr{G}):={R}^i\pi_{A \,
*}(\mathscr{P} \otimes \pi^*_{\wA}\mathscr{G}).
\end{equation*}

By the Base Change Theorem (see \cite[Chapter II]{Mum70}) it follows
that $\FF$ satisfies IT of index $i$ if and only if it satisfies WIT of
index $i$ and $\Phi^{\mathscr{P}}(\mathscr{F})$ is locally free. In
particular, any non-degenerate line bundle $\LL$ of index $i$ on $A$
is an IT-sheaf of index $i$ and its Fourier-Mukai transform
$\Phi^{\mathscr{P}}(\mathscr{L})$ is a vector bundle of rank $h^i(A,
\, \LL)$ on $\widehat{A}$.

\begin{proposition}\emph{\cite{Mu81} \cite[Theorem
14.2.2]{BL04}} \label{prop.mukai} Let $\FF$ be a WIT-sheaf of
index $i$ on $A$. Then $\Phi^{\mathscr{P}}(\mathscr{F})$ is a
WIT-sheaf of index $g-i$ on $\widehat{A}$ and
\begin{equation*}
\Phi^{\mathscr{P}} \circ \Phi^{\mathscr{P}} (\FF) = (-1)_{A}^*\FF.
\end{equation*}
\end{proposition}

\begin{remark} \label{rem.FM}
In general the Fourier-Mukai transform induces an equivalence of
categories between the two derived categories $D(A)$ and
$D(\widehat{A})$, such that for all $\mathscr{F} \in D(A)$ and
$\mathscr{G} \in D(\wA)$ one has
\begin{equation*}
\Phi^{\mathscr{P}} \circ \Phi^{\mathscr{P}}(\mathscr{F}) =
(-1)_{A}^*(\mathscr{F}) [-g] \quad \textrm{and} \quad
\Phi^{\mathscr{P}} \circ \Phi^{\mathscr{P}}(\mathscr{G}) =
(-1)_{\widehat{A}}^*(\mathscr{G}) [-g],
\end{equation*}
where  $[-g]$ means ``shift the complex $g$ places to the right".
When $\FF$ is a $WIT$-sheaf then the complex
$\Phi^{\mathscr{P}}(\mathscr{F})$ can be identified with a vector
bundle, since it is different from zero at most at one place.
\end{remark}

\begin{corollary}\label{cor.menouno}
Let $\FF$ be a WIT-sheaf on $A$. Then
\begin{equation*}
\Phi^{\mathscr{P}} ((-1)_A^* \FF) =
(-1)_{\widehat{A}}^*\Phi^{\mathscr{P}} (\FF).
\end{equation*}
\end{corollary}
\begin{proof}
Set $\mathscr{F}':=\Phi^{\mathscr{P}} ((-1)_A^* \FF)$; then by
Proposition \ref{prop.mukai}
\begin{equation*}
\FF=\Phi^{\mathscr{P}} \circ \Phi^{\mathscr{P}}((-1)_A^*
\FF)=\Phi^{\mathscr{P}}(\mathscr{F}'),
\end{equation*}
hence
\begin{equation*}
\Phi^{\mathscr{P}}(\FF)=\Phi^{\mathscr{P}} \circ
\Phi^{\mathscr{P}}(\mathscr{F}')=(-1)_{\widehat{A}}^*
\mathscr{F'}=(-1)_{\widehat{A}}^* \Phi^{\mathscr{P}} ((-1)_A^* \FF).
\end{equation*}
\end{proof}

\subsection{The Hesse pencil and the family of its dual curves}\label{sub.hal}

In the sequel we will use some  classical facts about the Hesse
pencil of cubic curves, that are summarized here for the reader's
convenience. We follow the treatment given in \cite{AD09}.

The \emph{Hesse pencil} is the one-dimensional linear system of
plane cubic curves defined  by
\begin{equation*}
E_{t_0, \, t_1}\colon t_0(x^3+y^3+z^3)+t_1xyz=0, \quad [t_0 : \,
t_1] \in \mathbb{P}^1.
\end{equation*}
Its nine base points are the inflection points of any smooth curve
in the pencil. There are four singular members in the Hesse pencil
and each one is the union of three lines:
\begin{equation*}
\begin{array}{ll}
E_{0,\, 1} \colon           & xyz=0, \\
E_{1, \, -3} \colon          & (x+y+z)(x+\omega y+ \omega^2 z)(x+ \omega^2 y + \omega z)=0, \\
E_{1,\, -3\omega}\colon     & (x+\omega y + z)(x+ \omega^2 y + \omega^2 z)(x+y+ \omega z)=0, \\
E_{1, \, -3\omega^2} \colon & (x+\omega^2 y+ z)(x+ \omega y + \omega
z)(x+y+\omega^2 z)=0.
\end{array}
\end{equation*}
We call the singular members the \emph{triangles}.

The dual curve $\mathbf{B}_{m_0, \, 3m_1}$ of a smooth member
$E_{m_0,3m_1}$ of the Hesse pencil is a plane curve of degree six
with nine cusps, whose equation in the dual plane
$(\mathbb{P}^{2})^{\vee}$ is
\begin{equation}\label{eq.dual.Hesse}
\begin{split}
& m^4_0(X^6+Y^6+Z^6)-m_0(2m^3_0+32m^3_1)(X^3Y^3+X^3Z^3+Y^3Z^3) \\
-24 & m^2_0m^2_1XYZ(X^3+Y^3+Z^3)-(24m^3_0m_1+48m^4_1)X^2Y^2Z^2=0.
\end{split}
\end{equation}
Note that the dual of a triangle becomes a triangle taken with
multiplicity two. For any pair $(m_0, \, m_1)$ there is a unique
cubic $\mathbf{C}_{m_0,\, 3m_1}$ passing through the nine cusps of
$\mathbf{B}_{m_0,\, 3m_1}$. The general element  of the pencil
generated by $\mathbf{B}_{m_0, \, 3m_1}$ and $2 \mathbf{C}_{m_0,\,
3m_1}$  is an irreducible curve of degree $6$ with nine nodes at
the nine cusps of $\mathbf{B}_{m_0,\, 3m_1}$. Such a pencil is
called the \emph{Halphen pencil} associated with
$\mathbf{B}_{m_0,\, 3m_1}$, see \cite{H1882}.

\section{The construction} \label{sec.constr}
The aim of this section and the next one is to prove the main result
of the paper, namely the following

\begin{theorem}\label{theo.main}
There exists a $4$-dimensional family $\mathcal{M}_{\Phi}$ of
minimal surfaces of general type with $p_g=q=2$ and $K^2=6$ such
that, for the general element $\widehat{S} \in \mathcal{M}_{\Phi}$,
the canonical class $K_{\widehat{S}}$ is ample and the Albanese map
$\hat{\alpha}\colon \widehat{S} \longrightarrow \widehat{A}$ is a
finite cover of degree $4$.

The Tschirnhausen bundle $\EE^{\vee}$ associated with $\hat{\alpha}$
is isomorphic to $\Phi^{\mathscr{P}}(\mathscr{L}^{-1})$, where
$\mathscr{L}$ is a polarization of type $(1,\,3)$ on $A$.

The family $\mathcal{M}_{\Phi}$ provides an irreducible component of
the moduli space $\mathcal{M}_{2, \, 2, \,6}^{\emph{can}}$ of
canonical models of minimal surfaces of general type with $p_g=q=2$
and $K^2=6$. Such a component is generically smooth.
\end{theorem}

We first outline the main idea of our construction, which is
inspired by the one used in \cite{CH06} in the simpler case where
the Albanese map has degree $3$. In order to build a quadruple
cover $\hat{\alpha} \colon \widehat{S} \lr \widehat{A}$ with
Tschirnhausen bundle $\EE^{\vee}$, we first build a quadruple
cover $\alpha \colon S \lr A$ with Tschirnhausen bundle
$\phi_{\LL^{-1}}^*\EE^{\vee}$ and then we identify those covers of
this type that descend to a quadruple cover $\hat{\alpha} \colon
\widehat{S} \lr \widehat{A}$ (Propositions \ref{prop.dis.sec} and
\ref{prop.tot.decomp}). Furthermore we prove that for a general
such a cover the surface $\widehat{S}$ is smooth
 and we compute its invariants (Propositions \ref{prop.sing.covers} and
\ref{prop.num.invariants.cover}). Finally, in Section
\ref{sec.moduli} we examine the subset $\mathcal{M}_{\Phi}$ of the
moduli space corresponding to our surfaces.
\smallskip

We start by considering a $(1, \,3)$-polarized abelian surface $(A,
\, \LL)$.  For all $\mathscr{Q} \in \textrm{Pic}^0(A)$ we have
\begin{equation*}
h^0(A, \, \LL \otimes \mathscr{Q} )=3, \quad h^1(A, \, \LL \otimes
\mathscr{Q})=0, \quad h^2(A, \, \LL \otimes \mathscr{Q})=0,
\end{equation*}
so by Serre duality the line bundle $\LL^{-1}$ satisfies IT of index
$2$. Its Fourier-Mukai transform $\mathscr{E}^{\vee}:=
\Phi^{\mathscr{P}}(\LL^{-1})$ is a rank $3$ vector bundle on $\wA$
which satisfies WIT of index $0$, see Proposition \ref{prop.mukai}.
The isogeny
\begin{equation*}
\phi:=\phi_{\LL^{-1}}\colon A \longrightarrow \wA, \quad
\phi(x)=t_x^*\LL^{-1} \otimes \LL
\end{equation*}
has kernel $K(\LL^{-1})=K(\LL) \cong (\ZZ/3 \ZZ)^2$ and  by
\cite[Proposition 3.11]{Mu81} we have
\begin{equation}\label{eq.mukai}
\phi^* \EE^{\vee} = \LL \oplus \LL \oplus \LL .
\end{equation}
Since $\Phi^{\mathscr{P}}(\EE^{\vee})= \Phi^{\mathscr{P}} \circ
\Phi^{\mathscr{P}}(\LL^{-1})=(-1)_A^* \LL^{-1}$ is locally free, it
follows that $\EE^{\vee}$ is actually a IT-sheaf of index $0$. By
\eqref{eq.mukai} we have
\begin{equation} \label{eq-chern.phi.E}
c_1(\phi^*\EE^{\vee})=3L, \quad c_2(\phi^*\EE^{\vee})=18,
\end{equation}
hence
\begin{equation} \label{eq-chern.E}
c_1(\EE^{\vee})=\widehat{L}, \quad c_2(\EE^{\vee})=2,
\end{equation}
where $\widehat{L}$ is a polarization of type $(1, \, 3)$ on
$\widehat{A}$. Therefore Hirzebruch-Riemann-Roch yields
$\chi(\widehat{A}, \, \EE^{\vee})=1$, which in turn implies
\begin{equation} \label{eq.cohom.EE^vee}
h^0(\widehat{A}, \, \EE^{\vee})=1, \quad h^1(\widehat{A}, \,
\EE^{\vee})=0, \quad h^2(\widehat{A}, \, \EE^{\vee})=0.
\end{equation}
Now we want to construct a quadruple cover $\widehat{\alpha} \colon
\widehat{S} \longrightarrow \widehat{A}$ with Tschirnhausen bundle
$\EE^{\vee}$. By \cite{HM99} it suffices to find those totally
decomposable elements in
\begin{equation*} H^0 \big(A, \, \bigwedge^2 S^2
(\phi^*\EE^{\vee}) \otimes \bigwedge^3 (\phi^* \EE) \big) \cong
H^0(A, \, \LL)^{\oplus 15}
\end{equation*}
that are $K(\mathscr{L}^{-1})$-invariant and therefore descend to
elements in $H^0 \big(\widehat{A}, \, \bigwedge^2 S^2 \EE^{\vee}
\otimes \bigwedge^3 \EE \big)$. This will be done in Propositions
\ref{prop.dis.sec} and \ref{prop.tot.decomp} below.

Let us consider the Heisenberg group
\begin{equation*}
\mathscr{H}_3:=\{(k, \, t, \, l) \mid k \in \CC^* , \, t \in
\ZZ/3\ZZ , \, l \in \widehat{\ZZ/3\ZZ} \},
\end{equation*}
whose group law is
\begin{equation*}
(k, \, t, \, l)\cdot (k', \, t', \, l')=(kk'l'(t), \, t+t', \,
l+l').
\end{equation*}

By \cite[Chapter 6]{BL04} there exists a canonical representation,
known as the Schr\"odinger representation, of $\mathscr{H}_3$ on
$H^0(A, \, \LL)$, where the latter space is identified with the
vector space $V:=\CC(\ZZ/3\ZZ)$ of all complex-valued functions on
the finite group $\ZZ/3\ZZ$. Such an action is given by
\begin{equation*}
(k, \, t, \, l)f(x)=kl(x)f(t+x).
\end{equation*}
Let $\{X, \, Y, \, Z\}$ be the basis  of $H^0(A, \, \LL)$
corresponding to the characteristic functions of $0, \, 1, \, 2$ in
$V$.

\begin{proposition} \label{prop.X.Y.Z-art}
For a general choice of the pair $(A, \, \mathscr{L})$, the three
effective divisors defined by $X$, $Y$, $Z \in H^0(A, \,
\mathscr{L})$ are smooth and intersect transversally. \
\end{proposition}
\begin{proof}
See Appendix $1$, in particular Proposition \ref{prop.X.Y.Z}.
\end{proof}

\begin{proposition}\label{prop.dis.sec}
The $K(\mathscr{L}^{-1})$-invariant subspace of $H^0 \big(A, \,
\bigwedge^2 S^2 (\phi^*\EE^{\vee}) \otimes \bigwedge^3 (\phi^* \EE)
\big)$ can be identified with
\begin{equation}\label{eq.dis.sec}
\begin{split}
 \phi^*H^0 \big(\widehat{A}, \, \bigwedge^2 S^2 \EE^{\vee} \otimes \bigwedge^3 \EE \big)  = & \{(aZ, \, bY, \, cY, \, dX, \, -cZ, \, eX, \, -bX, \\
 & -eZ, \, -dY, \, -dZ, \, eY, \, -aX, \, aY, \, cX, \, -bZ) \mid  \\ & a, \, b, \, c, \, d, \, e \in \CC\} \subset H^0(A, \, \LL)^{\oplus 15}.
\end{split}
\end{equation}
\end{proposition}
\begin{proof}
With respect to the basis $\{X, \, Y, \, Z\}$, the Schr\"odinger
representation can be written as
\begin{equation*}
\begin{array}{lll}
(1, \, 0, \, 0) \mapsto \left( \begin{matrix}
1&0&0\\
0&1&0\\
0&0&1
\end{matrix}\right), &
(1, \, 0, \, 1) \mapsto \left( \begin{matrix}
0&1&0\\
0&0&1\\
1&0&0
\end{matrix}\right), &
(1, \, 0, \, 2) \mapsto \left( \begin{matrix}
0&0&1\\
1&0&0\\
0&1&0
\end{matrix}\right),
\end{array}
\end{equation*}

\begin{equation*}
\begin{array}{lll}
(1, \, 1, \, 0) \mapsto \left( \begin{matrix}
1&0&0\\
0&\omega &0\\
0&0& \omega^2
\end{matrix}\right), &
(1, \, 1, \, 1) \mapsto \left( \begin{matrix}
0&1&0\\
0&0&\omega \\
\omega^2 &0&0
\end{matrix}\right), &
(1, \, 1, \, 2) \mapsto \left( \begin{matrix}
0&0&1\\
\omega &0&0\\
0& \omega^2 &0
\end{matrix}\right),
\end{array}
\end{equation*}

\begin{equation*}
\begin{array}{lll}
(1, \, 2, \, 0) \mapsto \left( \begin{matrix}
1&0&0\\
0&\omega^2 &0\\
0&0& \omega
\end{matrix}\right), &
(1, \, 2, \, 1) \mapsto \left( \begin{matrix}
0&1&0\\
0&0&\omega^2 \\
\omega &0&0
\end{matrix}\right), &
(1, \, 2, \, 2) \mapsto \left( \begin{matrix}
0&0&1\\
\omega^2 &0&0\\
0& \omega &0
\end{matrix}\right),
\end{array}
\end{equation*}

where $\omega:=e^{2\pi i/3}$. There is an induced representation of
$\mathscr{H}_3$ on the $45$-dimensional vector space $\bigwedge^2
S^2 V^{\vee} \otimes \bigwedge^3 V \otimes V$, and a long but
straightforward computation shows that this gives in turn a
representation of $K(\mathscr{L}^{-1}) \cong (\ZZ/3\ZZ)^2$. Denoting
by $\{\widehat{X}$, $\widehat{Y}$, $\widehat{Z}\}$ the dual basis of
$\{X, Y, Z\}$, one checks that the space of
$K(\mathscr{L}^{-1})$-invariant vectors in $\bigwedge^2 S^2
V^{\vee} \otimes \bigwedge^3 V \otimes V$ has dimension $5$ and is
generated by
\begin{equation*}
\begin{split}
v_1 = -(\widehat{X} \widehat{Z} \wedge \widehat{Z}^2) \otimes (X \wedge Y \wedge Z )\otimes X & + (\widehat{Y}^2 \wedge  \widehat{Y} \widehat{Z}) \otimes (X \wedge Y \wedge Z )\otimes Y \\ & +(\widehat{X}^2 \wedge  \widehat{X} \widehat{Y}) \otimes (X \wedge Y \wedge Z )\otimes Z,  \\
v_2 = -(\widehat{X} \widehat{Y} \wedge \widehat{Y}^2) \otimes (X \wedge Y \wedge Z )\otimes X & + (\widehat{X}^2 \wedge  \widehat{X} \widehat{Z}) \otimes (X \wedge Y \wedge Z )\otimes Y \\ & +(\widehat{Z}^2 \wedge  \widehat{Y} \widehat{Z}) \otimes (X \wedge Y \wedge Z )\otimes Z, \\
v_3 = (\widehat{Y}^2 \wedge \widehat{Z}^2) \otimes (X \wedge Y \wedge Z )\otimes X & + (\widehat{X}^2 \wedge  \widehat{Y}^2) \otimes (X \wedge Y \wedge Z )\otimes Y \\ & -(\widehat{X}^2 \wedge  \widehat{Z}^2) \otimes (X \wedge Y \wedge Z )\otimes Z, \\
v_4 = (\widehat{X}^2 \wedge \widehat{Y} \widehat{Z}) \otimes (X \wedge Y \wedge Z )\otimes X & - (\widehat{X} \widehat{Y}  \wedge  \widehat{Z}^2) \otimes (X \wedge Y \wedge Z )\otimes Y \\ & -(\widehat{X} \widehat{Z} \wedge  \widehat{Y}^2) \otimes (X \wedge Y \wedge Z )\otimes Z, \\
v_5 = (\widehat{X} \widehat{Y} \wedge \widehat{X} \widehat{Z}) \otimes (X \wedge Y \wedge Z )\otimes X & + (\widehat{X} \widehat{Z}  \wedge  \widehat{Y} \widehat{Z}) \otimes (X \wedge Y \wedge Z )\otimes Y \\ & -(\widehat{X} \widehat{Y} \wedge  \widehat{Y} \widehat{Z}) \otimes (X \wedge Y \wedge Z )\otimes Z. \\
\end{split}
\end{equation*}
Therefore any $K(\mathscr{L}^{-1})$-invariant vector can be written
as $av_1+bv_2+cv_3+dv_4+ev_5$ for some $a$, $b$, $c$, $d$, $e \in
\mathbb{C}$. Now the claim follows by choosing for $\bigwedge^2 S^2
V^{\vee}$ the basis
\begin{equation*}
\begin{array}{ccccc}
\{ \widehat{X}^2 \wedge \widehat{X}\widehat{Y}, & \widehat{X}^2 \wedge \widehat{X}\widehat{Z}, &  \widehat{X}^2 \wedge \widehat{Y}^2, &  \widehat{X}^2 \wedge \widehat{Y}\widehat{Z}, &  \widehat{X}^2 \wedge \widehat{Z}^2, \\
\widehat{X}\widehat{Y} \wedge \widehat{X}\widehat{Z}, & \widehat{X}\widehat{Y} \wedge \widehat{Y}^2, &  \widehat{X}\widehat{Y} \wedge \widehat{Y}\widehat{Z}, & \widehat{X}\widehat{Y} \wedge \widehat{Z}^2, &  \widehat{X}\widehat{Z} \wedge \widehat{Y}^2, \\
\widehat{X}\widehat{Z} \wedge \widehat{Y}\widehat{Z}, &
\widehat{X}\widehat{Z} \wedge \widehat{Z}^2, & \widehat{Y}^2 \wedge
\widehat{Y}\widehat{Z}, &  \widehat{Y}^2 \wedge \widehat{Z}^2, &
\widehat{Y}\widehat{Z} \wedge \widehat{Z}^2 \}.
\end{array}
\end{equation*}
\end{proof}
\begin{proposition}\label{prop.tot.decomp}  A section $\eta \in \phi^*H^0 \big(\widehat{A}, \, \bigwedge^2 S^2 \EE^{\vee} \otimes \bigwedge^3 \EE \big)$ corresponding to a vector as in \eqref{eq.dis.sec}
is totally decomposable if and only if
\begin{itemize}
\item[$\boldsymbol{(i)}$] $b=-a$, $d=0$ and $e=-a^2/c$, or
\item[$\boldsymbol{(ii)}$] $a=b=d=e=0$, or
\item[$\boldsymbol{(iii)}$] $a=b=c=d=0$.
\end{itemize}
In other words, the totally decomposable sections $($up to a
multiplicative constant$)$ are in one-to-one correspondence with the
points of the smooth conic of equations
\begin{equation*}
a^2+ce=0, \quad b+a=0, \quad d=0
\end{equation*}
in the projective space $\mathbb{P}^4$ with homogeneous coordinates
$[a:b:c:d:e]$.
\end{proposition}
\begin{proof}
Proposition \ref{prop.dis.sec} allows one to identify the building
data $c_{ij}$ in \cite[p. 7]{HM99} as follows:
\begin{equation}\label{eq.cij}
\begin{array}{lllll}
c_{12}=aZ, & c_{13}=bY, & c_{14}=cY, & c_{15}=dX, & c_{16}=-cZ, \\
c_{23}=eX, & c_{24}=-bX, & c_{25}=-eZ, & c_{26}=-dY, & c_{34}=-dZ, \\
c_{35}=eY, & c_{36}=-aX, & c_{45}=aY, & c_{46}=cX, & c_{56}=-bZ.
\end{array}
\end{equation}
By \cite[Theorem 3.1]{HM99} the corresponding section $\eta\in
\phi^*H^0 \big(\widehat{A}, \, \bigwedge^2 S^2 \EE^{\vee} \otimes
\bigwedge^3 \EE \big)$ is totally decomposable if and only if the
$c_{ij}$ satisfy the Pl\"ucker relations:
\begin{equation}\label{eq.plu}
\begin{array}{ccc}
c_{12}c_{34}-c_{13}c_{24}+c_{14}c_{23}=0, & c_{12}c_{35}-c_{13}c_{25}+c_{15}c_{23}=0, & c_{12}c_{36}-c_{13}c_{26}+c_{16}c_{23}=0, \\
c_{12}c_{45}-c_{14}c_{25}+c_{15}c_{24}=0, & c_{12}c_{46}-c_{14}c_{26}+c_{16}c_{24}=0, & c_{12}c_{56}-c_{15}c_{26}+c_{16}c_{25}=0, \\
c_{13}c_{45}-c_{14}c_{35}+c_{15}c_{34}=0, & c_{13}c_{46}-c_{14}c_{36}+c_{16}c_{34}=0, & c_{13}c_{56}-c_{15}c_{36}+c_{16}c_{35}=0, \\
c_{14}c_{56}-c_{15}c_{46}+c_{16}c_{45}=0, & c_{23}c_{45}-c_{24}c_{35}+c_{25}c_{34}=0, & c_{23}c_{46}-c_{24}c_{36}+c_{26}c_{34}=0, \\
c_{23}c_{56}-c_{25}c_{36}+c_{26}c_{35}=0, & c_{24}c_{56}-c_{25}c_{46}+c_{26}c_{45}=0, & c_{34}c_{56}-c_{35}c_{46}+c_{36}c_{45}=0. \\
\end{array}
\end{equation}
Substituting in \eqref{eq.plu} the values given in \eqref{eq.cij} we
get the result.
\end{proof}
%
%
%

Let us take now a point $[a: c] \in \mathbb{P}^1$.  By Theorem
\ref{theo.mir} and Proposition \ref{prop.tot.decomp} there exists a
quadruple cover $\alpha \colon S \longrightarrow A$ induced by the
totally decomposable section $\eta\in \phi^*H^0 \big(\widehat{A}, \,
\bigwedge^2 S^2 \EE^{\vee} \otimes \bigwedge^3 \EE \big)$
represented by the point $[a: -a:  c:  0: -a^2/c] \in \mathbb{P}^4$;
note that cases $\boldsymbol{(ii)}$ and $\boldsymbol{(iii)}$ in
Proposition \ref{prop.tot.decomp}  correspond to  $[a:c]=[0:1]$ and
$[a:c]=[1:0]$, respectively. By construction the cover $\alpha$ is
$K(\mathscr{L}^{-1})$-equivariant, so it induces a quadruple cover
$\hat{\alpha} \colon \widehat{S} \longrightarrow \widehat{A}$ that
fits into a commutative diagram
\begin{equation} \label{dia.projs}
\begin{xy}
\xymatrix{
S  \ar[d]_{\alpha} \ar[rr]^{\psi} & & \widehat{S} \ar[d]^{\widehat{\alpha}} \\
A   \ar[rr]^{\phi_{\mathscr{L}^{-1}}} & & \widehat{A},  \\
  }
\end{xy}
\end{equation}
where $\psi$ is an \'etale $(\mathbb{Z}/3 \mathbb{Z})^2$-cover.

We denote by $B$ and $R$ the branch divisor and the ramification
divisor of $\alpha \colon S \lr A$ and by $\widehat{B}$ and
$\widehat{R}$ those of $\hat{\alpha} \colon \wS \lr \wA$.

Following \cite{HM99}, over an affine open subset $U \subset A$ we
can describe the quadruple cover $\alpha\colon S \longrightarrow A$
as
\begin{equation*}
\textrm{Spec} \, \frac{\mathscr{O}_{U}[u, \,v, \, w]}{(F_1, \ldots ,
\, F_6)},
\end{equation*}
with
\begin{equation}\label{eq.qua.cover}
\begin{split}
F_1 &=u^2 -(a_1u+a_2v+a_3w+b_1), \\
F_2 &=uv -(a_4u+a_5v+a_6w+b_2), \\
F_3 &=uw -(a_7u+a_8v+a_9w+b_3),\\
F_4 &=v^2 -(a_{10}u+a_{11}v+a_{12}w+b_4), \\
F_5 &=vw - (a_{13}u+a_{14}v+a_{15}w+b_5), \\
F_6 &=w^2 -(a_{16}u+a_{17}v+a_{18}w+b_6).\\
\end{split}
\end{equation}
Here
\begin{equation*}
\begin{array}{lll}
a_{1}=\frac{1}{2}c_{23}, & a_{2}=-c_{13}, & a_{3}=c_{12}, \\ a_{4}=\frac{1}{4}c_{25}-\frac{1}{2}c_{34}, & a_{5}=-\frac{1}{2}c_{15}-\frac{1}{4}c_{23}, & a_{6}=c_{14}, \\
a_{7}=\frac{1}{2}c_{26}-\frac{1}{4}c_{35}, & a_{8}=-c_{16},  &
a_{9}=\frac{1}{2}c_{15}-\frac{1}{4}c_{23}, \\  a_{10}=c_{45},
& a_{11}=-\frac{1}{2}c_{25}, & a_{12}=c_{24}, \\
a_{13}=c_{46}, & a_{14}=-\frac{1}{2}c_{26}-\frac{1}{4}c_{35}, &
a_{15}=\frac{1}{4}c_{25}+\frac{1}{2}c_{34}, \\  a_{16}=c_{56}, &
a_{17}=-c_{36}, & a_{18}=\frac{1}{2}c_{35}
\end{array}
\end{equation*}
and
\begin{equation*}
\begin{array}{l}
b_1 =-a_1a_5+a_2a_4-a_2a_{11}-a_3a_{14}+a_5^2+a_6a_8, \\
b_2 =a_2a_{10}+a_3a_{13}-a_4a_5-a_6a_7, \\
b_3 =a_2a_{13}+a_3a_{16}-a_4a_8-a_7a_9, \\
b_4 =-a_1a_{10}+a_4^2-a_4a_{11}+a_5a_{10}+a_6a_{13}-a_7a_{12}, \\
b_5 =-a_5a_{13}+a_8a_{10}+a_{12}a_{17}-a_{14}a_{15}, \\
b_6 =-a_1a_{16}-a_4a_{17}+a_7^2-a_7a_{18}+a_8a_{13}+a_9a_{16}, \\
\end{array}
\end{equation*}
where the $c_{ij}$ are given in \eqref{eq.cij}.

Using the \verb|MAGMA| script in the Appendix 2 we can check that
the branch locus $B$ of $\alpha\colon S \longrightarrow A$ is the
element in $|6L|$ of equation
\begin{equation}\label{eq.branch.S}
\begin{split}
a^8c^4(X^6+Y^6+Z^6)& + a^2c\bigg(-\frac{4}{27}a^9-\frac{2}{9}a^6c^3-\frac{64}{9}a^3c^6 +\frac{256}{27}c^9\bigg)(X^3Y^3+X^3Z^3+Y^3Z^3)\\
                     & + a^4c^2\bigg(-\frac{2}{3}a^6+\frac{16}{3}a^3c^3-\frac{32}{3}c^6\bigg)XYZ(X^3+Y^3+Z^3) \\
                     & +\bigg(-\frac{1}{27}a^{12}-\frac{92}{27}a^9c^3+\frac{112}{9}a^6c^6+\frac{256}{27}a^3c^9-\frac{256}{27}c^{12}\bigg)X^2Y^2Z^2=0.
\end{split}
\end{equation}
We can also see \eqref{eq.branch.S} as the equation of a sextic
curve in the dual projective plane $(\mathbb{P}^{2})^{\vee}$ with
homogeneous coordinates $[X:Y:Z]$; we shall denote such a curve by
$\mathbf{B}$. Varying the point $[a:c] \in \mathbb{P}^1$, the curves
$\mathbf{B}$ form a (non-linear) pencil in $\mathbb{P}^{2 \vee}$,
which turns out to be the pencil of the dual curves of members of
the Hesse pencil. In fact, $\mathbf{B}$ is precisely the dual of the
curve $E_{m_0,\, 3m_1}$ with $m_0=a^2c$ and $m_1=\frac{1}{6} a^3 -
\frac{2}{3}c^3$, see Subsection \ref{sub.hal}.

\begin{remark} \label{rem.cuspidal.sextic}
The $\mathscr{H}_3$-equivariant sextic $\mathbf{B}$ also appears,
for some special choices of the pair $(A, \, \mathscr{L})$, as a
component of the branch locus of the $6$-fold cover
$\varphi_{|\mathscr{L}|} \colon A \to \mathbb{P}H^0(A, \,
\mathscr{L})^{\vee}$. See for instance \emph{\cite{BL94}} and
\emph{\cite{Cas99}}.
\end{remark}

The group $K(\mathscr{L}^{-1})$ acts on $\mathbf{B}$ and this
induces an action on $B$. The quotient of $B$ by this action is a
curve $\widehat{B} \in |2\widehat{L}|$, where $\widehat{L}$ is the
polarization of type $(1, \, 3)$ on $\wA$ appearing in
\eqref{eq-chern.E}; the curve $\widehat{B}$ is precisely the branch
locus of the quadruple cover $\hat{\alpha} \colon \widehat{S}
\longrightarrow \widehat{A}$.

Now set
\begin{equation*}
\begin{split}
T_1 &:=\{ [1:1], \, [\omega:1], \, [\omega^2:1], \, [1:0] \}, \\
T_2 &:=\{ [-2:1], \, [-2\omega:1], \, [-2\omega^2:1], \, [0:1] \}.
\end{split}
\end{equation*}

If $[a : c] \notin T_1 \cup T_2$ then $\mathbf{B}$ has nine ordinary
cusps as only singularities; since $X, \,Y, \,Z \in H^0(A, \,
\mathscr{L})$ and $L^2=6$, by Proposition \ref{prop.X.Y.Z-art} it
follows that, for a general choice of the pair $(A, \,
\mathscr{L})$, the curve $B$ has $54$ ordinary cusps as only
singularities. In this case by a \verb|MAGMA| calculation (see again
the script in Appendix 2) shows that $S$ is smooth; since $\psi
\colon S \longrightarrow \widehat{S}$ is \'etale, the surface
$\widehat{S}$ is also smooth. Moreover, the nine cusps of
$\mathbf{B}$ belong to a single $K(\LL^{-1})$-orbit and the
stabilizer of each cusp is the identity. Therefore the $54$ cusps of
$B$ fall in precisely six orbits and consequently $\widehat{B}$ is
an irreducible curve with six cusps as only singularities.

If $[a : c] \in  T_1 \cup T_2$ then $\mathbf{B}=2\mathbf{B}'$, where
$\mathbf{B}'$ is a triangle. Consequently, $B=2B'$, where $B'$ has
$18$ ordinary double points as only singularities. In this case the
\verb|MAGMA| script shows that $S$ is smooth if and only if $[a : c]
\notin T_1$; if instead $[a : c] \in T_1$  then the singular locus
of $S$ coincides with the preimage of $B$. Moreover, the group
$K(\LL^{-1})$ acts on the three sides of the triangle $\mathbf{B}'$;
each side has
 stabilizer isomorphic to $\ZZ/ 3 \ZZ$ and there is only one orbit; consequently,
$\widehat{B}=2\widehat{B}'$, where $\widehat{B}'$ is an irreducible
curve with two ordinary double points as only singularities.

In any case, the preimage of a general point in the branch divisor
$\widehat{B}$ consists of \emph{three} distinct points. This shows
that the quadruple cover $\alpha \colon \widehat{S} \lr \widehat{A}$
cannot factor through a double cover, hence it must be the Albanese
map of $\widehat{S}$.

Summing up, we have proven

\begin{proposition}\label{prop.sing.covers} Let $(A, \, \mathscr{L})$ be a general
$(1,\, 3)$-polarized abelian surface. Then the following holds:
\begin{itemize}
\item[$\boldsymbol{(i)}$] the surface $\widehat{S}$ is smooth
precisely when  $[a:c] \notin T_1$, whereas if $[a:c] \in
 T_1$ it has a $1$-dimensional singular locus;
\item[$\boldsymbol{(ii)}$] if $\widehat{S}$ is smooth and $[a:c]
\notin T_2$ then the branch locus $\widehat{B}$ of
$\hat{\alpha}\colon \widehat{S} \longrightarrow \widehat{A}$ is an
irreducible curve in $|2 \widehat{L}|$, where $\widehat{L}$ is a
polarization of type $(1, \, 3)$ on $\wA$, with six ordinary cusps
as only singularities; \item[$\boldsymbol{(iii)}$] if $\widehat{S}$
is smooth and $[a:c] \in T_2$ then $\widehat{B}=2\widehat{B}'$,
where $\widehat{B}'$ is an irreducible curve in $|\widehat{L}|$ with
two ordinary double points as only singularities;
\item[$\boldsymbol{(iv)}$] $\hat{\alpha} \colon \widehat{S} \lr
\widehat{A}$ is the Albanese map of $\widehat{S}$.
\end{itemize}
 \end{proposition}
%
%
%
We can now compute the invariants of $\widehat{S}$.
\begin{proposition}\label{prop.num.invariants.cover}
If $[a:c] \notin T_1$ and $\widehat{A}$ is general, then
$\widehat{S}$ is a minimal surface of general type with $p_g=q=2$
and $K^2=6$. The canonical class $K_{\widehat{S}}$ is ample and the
general element of $|K_{\widehat{S}}|$ is smooth and irreducible.
\end{proposition}
\begin{proof}
Using \eqref{eq.inv.hacca} and \eqref{eq.cohom.EE^vee}
 we obtain $p_g(\widehat{S})=q(\widehat{S})=2$. By Hurwitz formula we get
\begin{equation*}
K_{\widehat{S}}=\hat{\alpha}^*K_{\widehat{A}}+\widehat{R}=
\widehat{R},
\end{equation*}
so $\widehat{R} \in |K_{\widehat{S}}| $. If $\widehat{S}$ is general
we can check by our \verb|MAGMA| script that the restriction
$\hat{\alpha} \colon \widehat{R} \longrightarrow \widehat{B}$ is the
normalization map, that is $\widehat{R}$ is a smooth curve of genus
$7$ (see Proposition \ref{prop.sing.covers}). Thus the genus formula
gives
\begin{equation*}
2K^2_{\widehat{S}}=\widehat{R}(K_{\widehat{S}}+\widehat{R})=2g(\widehat{R})-2=12,
\end{equation*}
hence $K^2_{\widehat{S}}=6$. Moreover, by Bertini's theorem the
general element of $|K_{\widehat{S}}|$ is smooth and irreducible,
because $\widehat{R}$ is so. Since $\hat{\alpha} \colon \wS \lr \wA$
is a finite map onto an abelian variety, the surface $\wS$ contains
no rational curves. In particular $\wS$ is a minimal model and
$K_{\widehat{S}}$ is ample.
\end{proof}

\begin{remark} \label{rem-cusps-non-gen-pos}
The six cusps of the branch curve $\widehat{B}$ are not in general
position: in fact, by the results of Subsection
\emph{\ref{sub.hal}} it follows that there exists a unique element
$\widehat{C} \in |\widehat{L}|$ containing them. The general
element of the pencil generated by $\widehat{B}$ and $2
\widehat{C}$ is an irreducible curve in $|2 \widehat{L}|$ with six
nodes at the six cusps of $\widehat{B}$ and no further
singularities. Hence  the cuspidal curve $\widehat{B} \in |2
\widehat{L}|$ can be obtained as limit of nodal curves belonging
to the same linear system.
\end{remark}

In the sequel we denote by $\mathscr{A}_{1,\,3}$ the moduli space of
$(1,3)$-polarized abelian surfaces. It is a quasi-projective variety
of dimension $3$, see \cite[Chapter 8]{BL04}. The following result
will be used in the next section.

\begin{proposition} \label{prop.R.nonhyp}
For a general choice of $\widehat{A} \in \mathscr{A}_{1,3}$ and
$[a:c] \in \mathbb{P}^1$ the curve $\widehat{R}$ is not
hyperelliptic. In particular, for any line bundle $\mathscr{N}$ on
$\widehat{R}$ with $\deg (\mathscr{N}) =6$ we have $h^0(\widehat{R},
\, \mathscr{N}) \leq 3$.
\end{proposition}
\begin{proof}
This follows from a rigidity result of hyperelliptic curves on
general abelian varieties, see \cite{Pi89}. More precisely, one
 can show that the only hyperelliptic deformations of a (possibly
singular) hyperelliptic curve on a fixed simple abelian variety are
just  the translations. This implies that a linear system on a
general abelian surface contains at most a finite number of
hyperelliptic curves. In our situation, we have an equisingular,
$1$-dimensional family $\{\widehat{B} \}$ of cuspidal curves in the
linear system $|2 \widehat{L}|$, parametrized by the points $[a:c]
\in \mathbb{P}^1$. Then, if $\wA$ is simple, the general curve
$\widehat{B}$ is non-hyperelliptic, that is its normalization $\wR$
is non-hyperelliptic.

If $\mathscr{N}$ is any line bundle on $\wR$ with
$\deg(\mathscr{N})=6$, Riemann-Roch yields $h^0(\wR, \,
\mathscr{N})-h^1(\wR, \, \mathscr{N})=0$; in particular, if
$\mathscr{N}$ is non-special we obtain $h^0(\wR, \, \mathscr{N})=0$.
If instead $\mathscr{N}$ is special, since $\wR$ is
non-hyperelliptic
 Clifford's theorem implies
\begin{equation*}
h^0(\wR, \, \mathscr{N}) -1 < \frac{1}{2} \deg(\mathscr{N}),
\end{equation*}
that is $h^0(\wR, \, \mathscr{N}) \leq 3$.
\end{proof}


\section{The moduli space} \label{sec.moduli}

\begin{definition}
We denote by $\mathcal{M}_{\Phi}^0$ the family of canonical models
 $\widehat{X}$ of minimal surfaces of general type
 with $p_g=q=2$, $K^2 =6$ such that:
\begin{itemize}
\item[$\boldsymbol{(1)}$] the Albanese map $\hat{\alpha} \colon
\widehat{X} \to \widehat{A}$ is a finite, quadruple cover;
\item[$\boldsymbol{(2)}$] the Tschirnhausen bundle $\EE^{\vee}$
associated with $\hat{\alpha} \colon \widehat{X} \to \widehat{A}$ is
of the form $\EE^{\vee}=\Phi^{\mathscr{P}}(\mathscr{L}^{-1})$, where
$\mathscr{L}$ is a polarization of type $(1,\,3)$ on $A$.
\end{itemize}
\end{definition}

Notice that $\mathcal{M}_{\Phi}^0$ coincides with the family of
canonical models of surfaces $\widehat{S}$ constructed in Section
\ref{sec.constr}, and that such a family depends on four parameters
(three parameters from $\wA \in \mathscr{A}_{1,\,3}$ and one
parameter from $[a:c] \in \mathbb{P}^1$). More precisely, there is a
generically finite, dominant map
\begin{equation*}
\mathbf{P}^0 \to \mathcal{M}_{\Phi}^0,
\end{equation*}
where $\mathbf{P}^0$ is a Zariski-dense subset of a
$\mathbb{P}^1$-bundle over $\mathscr{A}_{1,\,3}$. Then
 $\mathcal{M}_{\Phi}^0$ is irreducible  and $\dim \mathcal{M}_{\Phi}^0=4$.

\begin{proposition} \label{prop.E.stable}
If $\wS \in \mathcal{M}_{\Phi}^0$ then the Tschirnhausen bundle
$\EE^{\vee}$ is stable with respect to the polarization
$\widehat{L}$. In particular it is simple, that is $H^0(\wA, \,
\EE \otimes \EE^{\vee})= \mathbb{C}$.
\end{proposition}
\begin{proof}
By definition of $\EE^{\vee}$ and since the Fourier-Mukai transform gives an equivalence of derived categories we have
\begin{equation*}
\label{eq.LdL} H^0(\wA, \, \EE \otimes \EE^{\vee}) \cong
\Hom_{\wA}(\EE^{\vee}, \EE^{\vee}) \cong \Hom_{A}(\mathscr{L}^{-1},
\mathscr{L}^{-1}) =\mathbb{C},
\end{equation*}
that is $\EE^{\vee}$ is simple.
 Then, by
\eqref{eq.mukai} and \cite[Exercise 2 page 476]{BL04}, it follows
that $\EE^{\vee}$ is semi-homogeneous. Finally, any simple,
semi-homogeneous vector bundle on an abelian variety is stable with
respect to any polarization, see \cite[Exercise 1 page 476]{BL04}.
\end{proof}

\begin{definition}
We denote by $\mathcal{M}_{\Phi}$ the closure of
$\mathcal{M}_{\Phi}^0$ in the moduli space
$\mathcal{M}_{2,\,6,\,6}^{\emph{can}}$ of canonical models of
minimal surfaces of general type with $p_g=q=2$, $K^2=6$.
\end{definition}

By the previous considerations, $\mathcal{M}_{\Phi}$ is irreducible
and $\dim \mathcal{M}_{\Phi}=4$. Now we want to prove that
$\mathcal{M}_{\Phi}$ provides an
 irreducible component of the moduli space $\mathcal{M}_{2,\, 2,
\, 6}^{\textrm{can}}$ and that such a component is generically
smooth. In order to do this, we must prove that for the general
surface $\widehat{S} \in \mathcal{M}_{\Phi}$ one has
$h^1(\widehat{S}, \, T_{\widehat{S}})= \dim \mathcal{M}_{\Phi}=4$.

By \cite[p. 262]{Se06} we have an exact sequence
\begin{equation}\label{seq.deform}
0 \longrightarrow T_{\widehat{S}} \stackrel{d
\hat{\alpha}}{\longrightarrow} \hat{\alpha}^*T_{\widehat{A}}
\longrightarrow \mathscr{N}_{\hat{\alpha}} \longrightarrow 0,
\end{equation}
where $\mathscr{N}_{\hat{\alpha}}$ is a coherent sheaf supported on
$\widehat{R}$ and called the normal sheaf of $\hat{\alpha}$;
 we denote by $\varphi= \hat{\alpha}|_{\widehat{R}} \colon \widehat{R} \to
\widehat{B} \subset \widehat{A}$ the normalization map of
$\widehat{B}$.

Let $\Delta \subset \widehat{R}$ be the divisor formed by all points
$p$ such that $\varphi(p)$ is a cusp of $\widehat{B}$, each counted
with multiplicity $1$, and let $i \colon \widehat{R} \to
\widehat{S}$ be the inclusion of $\widehat{R}$ in $\widehat{S}$;
then $\hat{\alpha} \circ i=\varphi$. By \cite[Section 7]{CF11} there
are two commutative diagrams
\begin{equation} \label{dia.1}
\xymatrix{
 & & & 0 \ar[d]& \\
 & 0 \ar[d] &  & \mathscr{O}_{\Delta} \ar[d]  & \\
0 \ar[r] & T_{\wR} \ar[r]^{d \varphi} \ar[d] & \varphi^*T_{\wA}
\ar[r] \ar@{=}[d] & \mathscr{N}_{\varphi}
\ar[r] \ar[d]& 0 \\
  & i^*T_{\wS} \ar[r]^{i^*(d \hat{\alpha})} & \varphi^*T_{\wA} \ar[r]  & i^* \mathscr{N}_{\hat{\alpha}} \ar[r] \ar[d] &
  0 \\
  & & & 0 & }
\end{equation}
and
\begin{equation} \label{dia.2}
\xymatrix{
 & & & 0 \ar[d]& \\
 & 0 \ar[d] &  & \mathscr{O}_{\Delta} \ar[d]  & \\
0 \ar[r] & T_{\wR} \ar[r]^{d \varphi} \ar[d] & \varphi^*T_{\wA}
\ar[r] \ar@{=}[d] & \mathscr{N}_{\varphi}
\ar[r] \ar[d]& 0 \\
  & T_{\wR} \otimes \oo_{\wR}(\Delta) \ar[r]^{d \varphi} & \varphi^*T_{\wA} \ar[r]  & \mathscr{N}'_{\varphi} \ar[r] \ar[d] &
  0, \\
  & & & 0 & }
\end{equation}
where $\mathscr{N}'_{\varphi}$ is a line bundle on $\wR$ satisfying
\begin{equation*} \label{eq.deg.Nphi}
\begin{split}
\deg (\mathscr{N}'_{\varphi})=\deg(\mathscr{N}_{\varphi})-\deg
(\Delta) & =\deg (\varphi^* T_{\wA}) - \deg (T_{\wR}) - \deg
(\Delta)\\ & = \deg(\oo_{\wR}^{\oplus 2}) + \deg (K_{\wR})- \deg
(\Delta)=6.
\end{split}
\end{equation*}
The fact that $\mathscr{N}_{\hat{\alpha}}$ is supported on
$\widehat{R}$, together with \eqref{dia.1} and \eqref{dia.2}, implies
that
\begin{equation*} \label{eq.hi(N)}
h^i(\wS, \, \mathscr{N}_{\hat{\alpha}})=h^i(\wR, \,
i^*\mathscr{N}_{\hat{\alpha}})=h^i(\wR, \, \mathscr{N}'_{\varphi}),
\quad i=0,\,1,\,2,
\end{equation*}
so for a general choice of the abelian variety $\wA \in
\mathscr{A}_{1,3}$ and of the point $[a:c] \in \mathbb{P}^1$ one
has
\begin{equation} \label{eq.dimNp}
h^0(\wS, \, \mathscr{N}_{\hat{\alpha}})=h^0(\wR, \,
\mathscr{N}'_{\varphi}) \leq 3,
\end{equation}
see Proposition \ref{prop.R.nonhyp}. Now we can prove the desired
result.

\begin{proposition} \label{prop.irred.comp}
If $\widehat{S}$ is a general element of $\mathcal{M}_{\Phi}$ then
$h^1(\widehat{S}, \, T_{\widehat{S}})=4$. Hence $\mathcal{M}_{\Phi}$
provides an irreducible component of the moduli space
$\mathcal{M}_{2,\, 2, \,6}^{\emph{can}}$ of canonical models of
minimal surfaces with $p_g=q=2$ and $K^2=6$. Such a component is
generically smooth, of dimension $4$.
\end{proposition}
\begin{proof}
Since $\dim \mathcal{M}_{\Phi} =4$, it is sufficient to show that
for a general choice of $\widehat{S}$ one has
\begin{equation} \label{h1N leq 4}
h^1(\wS, \, T_{\wS}) \leq 4.
\end{equation}

The surface $\widehat{S}$  is of general type, so we have
$H^0(\widehat{S}, \, T_{\widehat{S}})=0$ and \eqref{seq.deform}
yields the following exact sequence in cohomology
\begin{equation}\label{seq.deform.coh}
0 \longrightarrow H^0(\widehat{S}, \, \hat{\alpha}^*T_{\widehat{A}})
\longrightarrow H^0(\widehat{S}, \, \mathscr{N}_{\hat{\alpha}})
\longrightarrow  H^1(\widehat{S}, \, T_{\widehat{S}})
\stackrel{\varepsilon}{\longrightarrow} H^1(\widehat{S}, \,
\hat{\alpha}^*T_{\widehat{A}}).
\end{equation}
The same argument used in \cite[Section 3]{PP11} shows that the
image of $\varepsilon \colon H^0(\widehat{S}, \, T_{\widehat{S}})
\longrightarrow H^1(\widehat{S}, \, \hat{\alpha}^*T_{\widehat{A}})$
has dimension $3$ (this is essentially a consequence of the fact
that when one deforms $\widehat{S}$ the Albanese torus $\wA$ remains
algebraic). Then \eqref{h1N leq 4} follows from \eqref{eq.dimNp} and
\eqref{seq.deform.coh}.
\end{proof}

\begin{proposition} \label{rem-no-irr}
The general element $\widehat{S}$ of $\mathcal{M}_{\Phi}$ admits no
pencil $p \colon \widehat{S} \to T$ over a curve $T$ with $g(T) \geq
1$.
\end{proposition}
\begin{proof}
Assume that $\wA$ a is simple abelian surface. In this case the set
\begin{equation*}
V^1(\widehat{S}):=\{ \mathscr{Q} \in \textrm{Pic}^0(\widehat{S}) \;
| \; h^1(\widehat{S}, \, \mathscr{Q}^{\vee}) >0 \}
\end{equation*}
cannot contain any component of positive dimension, so $\widehat{S}$
does not admit any pencil over a curve $T$ with $g(T) \geq 2$, see
\cite[Theorem 2.6]{HP02}. If instead $g(T)=1$, the universal
property of the Albanese map yields a surjective morphism
$\widehat{A} \to T$, contradicting the fact that $\wA$ is simple.
\end{proof}

\section{Quadruple covers with simple Tschirnhausen bundle} \label{sec.simple}
\begin{proposition}\label{prop.simple}
Let $(\widehat{A}, \, \widehat{\mathscr{L}})$ be a $(1,3)$-polarized
abelian surface and $\EE^{\vee}$ be a \emph{simple} rank $3$ vector
bundle on $\widehat{A}$ with
\begin{equation*}
h^0(\widehat{A}, \, \EE^{\vee})=1, \quad h^1(\widehat{A}, \,
\EE^{\vee})=0, \quad h^2(\widehat{A}, \, \EE^{\vee})=0,
\end{equation*}
\begin{equation*}
c_1(\EE^{\vee})= \widehat{\mathscr{L}}, \quad  c_2(\EE^{\vee})=2.
\end{equation*}
Then there exists a polarization $\mathscr{L}$ of type $(1,3)$ on
$A$ such that
\begin{equation*}
\EE^{\vee}=\Phi^{\mathscr{P}}(\mathscr{L}^{-1}).
\end{equation*}
\end{proposition}
\begin{proof}
Since $\EE^{\vee}$ is simple and
$2c^2_1(\EE^{\vee})-6c_2(\EE^{\vee})=0$, by \cite[Corollary p.
249]{O71} and \cite[Ex. 2 p. 476]{BL04} there exist an abelian
surface $Z$, a line bundle $\mathscr{G}$ on $Z$ and an isogeny
$g\colon Z \longrightarrow \widehat{A}$ such that
$g^*\EE^{\vee}=\mathscr{G}^{\oplus 3}$. Hence we obtain
\begin{equation*}
3c_1(\mathscr{G})=c_1(g^*\EE^{\vee})= g^* c_1(\EE^{\vee}) =
g^*\widehat{\mathscr{L}},
\end{equation*}
that is $\mathscr{G}$ is ample. For any $\mathscr{Q} \in
\textrm{Pic}^0(\widehat{A})$ we have $g^*\mathscr{Q} \in
\textrm{Pic}^0(Z)$ and $g^*(\EE^{\vee} \otimes
\mathscr{Q})=(\mathscr{G} \otimes g^*\mathscr{Q})^{\oplus 3}$, so
the ampleness of $\mathscr{G}$ implies
\begin{equation*}
H^i(Z, \, g^*(\EE^{\vee} \otimes \mathscr{Q}))= H^i(Z, \,
\mathscr{G} \otimes g^*\mathscr{Q})^{\oplus 3}=0, \quad i=1, \, 2
\end{equation*}
for all  $\mathscr{Q} \in \textrm{Pic}^0(\widehat{A})$. Since $g$
is a finite map, we get
\begin{equation*}
H^i(\widehat{A}, \, \EE^{\vee} \otimes \mathscr{Q})\cong
g^*H^i(\widehat{A}, \, \EE^{\vee} \otimes \mathscr{Q}) \subseteq
H^i(Z, \, g^*(\EE^{\vee} \otimes \mathscr{Q}))=0, \quad i=1, \, 2
\end{equation*}
that is $\EE^{\vee}$ satisfies IT of index $0$. Since
$\textrm{rank}(\EE^{\vee})=3$ and $h^0(\widehat{A}, \,
\EE^{\vee})=1$, the Fourier-Mukai transform
$\Phi^{\mathscr{P}}(\EE^{\vee})$ is a line bundle of type $(1, \,
3)$ on $A$ that we denote by $\mathscr{M}^{-1}$. Therefore we have
\begin{equation*}
(-1)_{\widehat{A}}^*\EE^{\vee} = \Phi^{\mathscr{P}} \circ
\Phi^{\mathscr{P}}(\EE^{\vee}) =
\Phi^{\mathscr{P}}(\mathscr{M}^{-1}),
\end{equation*}
so $\EE^{\vee}=\Phi^{\mathscr{P}}((-1)_A^*\mathscr{M}^{-1})$ by
Corollary \ref{cor.menouno}. Setting
$\mathscr{L}^{-1}:=(-1)_A^*\mathscr{M}^{-1}$ we are done.
\end{proof}

The following corollary can be seen as a converse of Proposition
\ref{prop.E.stable}.

\begin{corollary} \label{cor.simple}
Let $\widehat{S}$ be a surface of general type with $p_g=q=2$,
$K_{\wS}^2=6$ such that the Albanese map of its canonical model
$\widehat{X}$ is a finite, quadruple cover $\hat{\alpha} \colon
\widehat{X} \to \wA$, where $\wA$ is a $(1, \, 3)$-polarized abelian
surface. Assume in addition that the Tschirnhausen bundle
$\mathscr{E}^{\vee}$ associated with $\hat{\alpha}$ is simple. Then
$\widehat{X}$ belongs to $\mathcal{M}_{\Phi}$.
\end{corollary}

\section{Two remarkable subfamilies of $\mathcal{M}_{\Phi}$}
\label{sec.two-sub}

\subsection{A $3$-dimensional family of surfaces contained in the singular locus of $\mathcal{M}_{2, \, 2, \,6}^{\textrm{can}}$ }
If we choose $[a:c]=[0:1]$ in the construction of Section
\ref{sec.constr}, by \eqref{eq.cij} we obtain
\begin{equation} \label{eq.cij.bidouble}
c_{12}=c_{13}=c_{15}=c_{23}=c_{24}=c_{25}=c_{26}=c_{34}=c_{35}=c_{36}=c_{45}=c_{56}=0.
\end{equation}
Then  \eqref{eq.qua.cover} implies that the local
 equations of the surface $S$ are
\begin{equation} \label{eq.bidouble}
\begin{array}{ll}
    u^2=YZ \\
    uv=Yw \\
    uw=Zv \\
    v^2=XY \\
    vw=Xu \\
    w^2=XZ
  \end{array}
\end{equation}
and this shows that the Albanese map $\alpha \colon S \to A$ is a
\emph{bidouble cover}, i.e. a Galois cover with Galois group
$(\mathbb{Z}/2 \ZZ)^2$. Conversely, relations
\eqref{eq.cij.bidouble} are precisely the conditions ensuring that
$\alpha \colon S \to A$ is a bidouble cover, see \cite[p.
25-27]{HM99}.

We use the theory of bidouble covers developed in \cite{Ca84}. The
cover $\alpha \colon S \to A$ is branched over the three divisors
$D_X$, $D_Y$, $D_Z \in |L|$ corresponding to the distinguished
sections $X$, $Y$, $Z \in H^0(A, \, \LL)$; for simplicity of
notation we write $D_1$, $D_2$, $D_3$ instead of $D_X$, $D_Y$,
$D_Z$, respectively. Then the building data of $\alpha \colon S \to
A$ consist of three line bundles $\mathscr{L}_1$, $\mathscr{L}_2$,
$\mathscr{L}_3$ on $A$, with $\mathscr{L}_i=\mathscr{O}_A(L_i)$,
such that
\begin{equation}
\begin{array}{ll}
2L_1 \cong D_2 + D_3 \\
2L_2 \cong D_1+ D_3  \\
2L_3 \cong D_1+ D_2
\end{array}
\end{equation}
and $\alpha_* \mathscr{O}_S = \mathscr{O}_A \oplus
\mathscr{L}_1^{-1} \oplus \mathscr{L}_2^{-1} \oplus
\mathscr{L}_3^{-1}$. On the other hand, by \eqref{eq.mukai} we
obtain $\alpha_* \mathscr{O}_S = \mathscr{O}_A \oplus \phi^* \
\mathscr{E}=\mathscr{O}_A \oplus \mathscr{L}^{-1} \oplus
\mathscr{L}^{-1} \oplus \mathscr{L}^{-1}$; therefore \cite{At56}
implies $\LL_1=\LL_2=\LL_3= \LL$.

By \cite[p. 497]{Ca84} there is an exact sequence
\begin{equation} \label{seq.obstructed}
\begin{split}
0  & \to H^0(S, \, \alpha^* T_A) \to \bigoplus_{i=1}^3
H^0(\mathscr{O}_{D_i}(D_i) \oplus \mathscr{O}_{D_i})
\stackrel{\partial}{\to} H^1(S, \, T_S) \\ &
\stackrel{\varepsilon}{\to} H^1(S, \,\alpha^* T_A) \to \cdots,
\end{split}
\end{equation}
whose meaning is the following: the image of $\partial$ are the
first-order deformations coming from the so called ``natural
deformations" (see \cite[Definition 2.8 p. 494]{Ca84}) and such
first order deformations are trivial if they arise from automorphisms
of $A$. Moreover, if one considers the map
\begin{equation*}
\partial' \colon \bigoplus_{i=1}^3
H^0(\mathscr{O}_A(D_i) \oplus \mathscr{O}_A) \to H^1(S, \, T_S),
\end{equation*}
obtained as the composition of the direct sum of the restriction
maps with $\partial$, then for any $\bigoplus _i(a_i \oplus
\delta_i) \in \bigoplus_i H^0(\mathscr{O}_A(D_i) \oplus
\mathscr{O}_A)$ the element $\partial'(\bigoplus_i(a_i \oplus
\delta_i)) \in H^1(S, \, T_S)$ is the Kodaira-Spencer class of the
corresponding natural deformation of $S$. More explicitly, taking
\begin{equation*}
a_i = \alpha_i X + \beta_i Y + \gamma_i Z,    \quad \alpha_i, \,
\beta_i, \, \gamma_i \in \mathbb{C}
\end{equation*}
the local equations of the associated natural deformation of the
bidouble cover \eqref{eq.bidouble} are
\begin{equation} \label{eq.bidouble.deformed}
\begin{array}{ll}
    u^2=(Y+ \alpha_3 X + \beta_3 Y + \gamma_3 Z+\delta_3 w)(Z+ \alpha_2 X + \beta_2 Y + \gamma_2 Z+\delta_2 v)  \\
    uv= (Y + \alpha_3 X + \beta_3 Y + \gamma_3 Z)w+ \delta_3 w^2  \\
    uw= (Z+ \alpha_2 X + \beta_2 Y + \gamma_2 Z)v+ \delta_2 v^2  \\
    v^2=(X+ \alpha_1 X + \beta_1 Y + \gamma_1 Z+\delta_1 u )(Y+ \alpha_3 X + \beta_3 Y + \gamma_3 Z+\delta_3 w)  \\
    vw= (X+\alpha_1 X + \beta_1 Y + \gamma_1 Z)u+ \delta_1 u^2  \\
    w^2=(X+\alpha_1 X + \beta_1 Y + \gamma_1 Z + \delta_1 u )(Z + \alpha_2 X + \beta_2 Y + \gamma_2
Z + \delta_2 v).
  \end{array}
\end{equation}

By using the restriction exact sequence
\begin{equation*}
0 \to H^0(A, \, \mathscr{O}_A) \to H^0(A, \, \mathscr{O}_A(D_i))
\stackrel{\rho_i}{\to} H^0(D_i, \, \mathscr{O}_{D_i}(D_i)) \to
H^1(A, \, \mathscr{O}_A) \to 0
\end{equation*}
we obtain
\begin{equation} \label{eq.HDi}
H^0(D_i, \, \mathscr{O}_{D_i}(D_i))= \textrm{im} \, \rho_i \oplus
H^1(A, \, \mathscr{O}_A),
\end{equation}
where
\begin{equation} \label{eq.im.rho}
\begin{split}
\textrm{im} \, \rho_1 & = \textrm{span}(Y/X, \, Z/X) \\
\textrm{im} \, \rho_2 & = \textrm{span}(X/Y, \, Z/Y) \\
\textrm{im} \, \rho_3 & = \textrm{span}(X/Z, \, Y/Z). \\
\end{split}
\end{equation}

There is an action of $(\mathbb{Z}/ 3 \mathbb{Z})^2=\langle r, \, s
\, | \, r^3=s^3=[r,s]=1 \rangle$ on $S$ given by
\begin{equation} \label{eq.action}
\begin{split}
r(u, \, v, \, w, \, X, \, Y, \, Z) & :=(u, \, \omega v, \, \omega^2
w, \, X, \, \omega^2 Y, \, \omega Z), \\
s(u, \, v, \, w, \, X, \, Y, \, Z) & := (v, \, w, \, u, \, Z, \, X,
\, Y)
\end{split}
\end{equation}
and the corresponding quotient map is precisely $\psi \colon S \to
\widehat{S}$; note that the Albanese map $\hat{\alpha} \colon
\widehat{S} \to \widehat{A}$ is \emph{not} a Galois cover. We shall
denote by $\chi_0$ the trivial character of $(\mathbb{Z}/ 3
\mathbb{Z})^2$ and by $\chi$ any non-trivial character; moreover, if
$V$ is a vector space with a $(\mathbb{Z}/ 3
\mathbb{Z})^2$-representation, then we indicate by $V^{\chi_0}$ the
invariant eigenspace and by $V^{\chi}$ the eigenspace relative to
$\chi$. For instance, since $q(A)=q(\widehat{A})$ and
$q(S)=q(\widehat{S})$, it follows
\begin{equation} \label{eq.invariant-sub-H1}
\begin{split}
H^i(A, \, \mathscr{O}_A)^{\chi_0}&=H^i(A, \, \mathscr{O}_A), \quad
 i=0, \, 1,\\
H^i(A,\, \mathscr{O}_A)^{\chi}&=0 \quad \chi \neq \chi_0, \quad
 i=0, \, 1,
\end{split}
\end{equation}
\begin{equation} \label{eq.invariant-sub-H1*}
\begin{split}
H^i(S, \, \mathscr{O}_S)^{\chi_0}&=H^i(S, \, \mathscr{O}_S), \quad
 i=0, \, 1,\\
H^i(S,\, \mathscr{O}_S)^{\chi}&=0 \quad \chi \neq \chi_0, \quad
 i=0, \, 1.
\end{split}
\end{equation}
Since $\alpha^*T_A = \mathscr{O}_S^{\oplus 2}$, we also have
\begin{equation} \label{eq.invariant-sub-H0}
\begin{split}
H^i(S, \, \alpha^*T_A)^{\chi_0}&=H^i(S, \, \alpha^*T_A), \quad
i=0, \, 1 \\
H^i(S, \, \alpha^*T_A)^{\chi}&=0 \quad \chi \neq \chi_0, \quad i=0,
\, 1.
\end{split}
\end{equation}

\begin{lemma} \label{prop.im.delta}
The action \emph{\eqref{eq.action}} induces a natural action of
$(\mathbb{Z}/ 3 \mathbb{Z})^2$ on $\emph{im}\, \partial$ such that
\begin{equation*}
\dim (\emph{im} \,  \partial)^{\chi_0} =5.
\end{equation*}
\end{lemma}
\begin{proof}
Using \eqref{eq.HDi}, \eqref{eq.im.rho} and
\eqref{eq.invariant-sub-H1}, by a tedious but straightforward
computation one proves that there is an induced action of
$(\mathbb{Z}/ 3 \mathbb{Z})^2$ on the vector space $\bigoplus_i
H^0(\mathscr{O}_{D_i}(D_i) \oplus \mathscr{O}_{D_i}) \cong
\mathbb{C}^{15}$, such that
\begin{equation*}
\dim \bigg( \bigoplus_i H^0(\mathscr{O}_{D_i}(D_i) \oplus
\mathscr{O}_{D_i}) \bigg)^{\chi_0}=7
\end{equation*}
whereas the eight remaining eigensheaves are all $1$-dimensional. By
\eqref{eq.invariant-sub-H0} the $2$-dimensional subspace $\ker
\partial = H^0(S, \, \alpha^* T_A)$ is contained in the invariant
eigenspace, so the claim follows.
\end{proof}

\begin{proposition} \label{prop.M.non.normal}
Assume that $\alpha \colon S \to A$ is a bidouble cover, that is
$[a:c]=[0:1]$. Then $h^1(\wS, \, T_{\wS})=8$.
\end{proposition}
\begin{proof}
By \eqref{seq.obstructed} we have
\begin{equation*}
H^1(S, \, T_S)=\textrm{im} \, \partial \oplus \textrm{im} \,
\varepsilon
\end{equation*}
and, as remarked in the proof of Proposition \ref{prop.irred.comp},
we can prove that $ \textrm{im} \, \varepsilon \subset H^1(S, \,
\alpha^* T_A)$ has dimension $3$. Using Lemma \ref{prop.im.delta}
and the fact that $H^1(S, \alpha^* T_A)$ is contained in the
invariant eigenspace, see \eqref{eq.invariant-sub-H0}, it follows
that the dimension of $H^1(S, \, T_S)^{\chi_0}$ equals $5+3=8$.
Finally, the $(\mathbb{Z}/3 \mathbb{Z})^2$-cover $\psi \colon S \to
\wS$ is unramified, so by \cite[Proposition 4.1]{Pa91} we infer
$H^1(S, \, T_S)^{\chi_0}=H^1(\wS, \, T_{\wS})$ and we are done.
\end{proof}

Since the irreducible component $\mathcal{M}_{\Phi}$ of
$\mathcal{M}_{2, \, 2, \,6}^{\textrm{can}}$ has dimension $4$,
Proposition \ref{prop.M.non.normal} implies
 that if $\alpha \colon S \to A$ is a bidouble cover then the canonical model
 of $\wS$ yields a singular point of the moduli space. So we obtain

\begin{corollary} \label{cor.sing.loc.M}
 The moduli space $\mathcal{M}_{2, \,
2, \,6}^{\emph{can}}$ contains a $3$-dimensional singular locus.
\end{corollary}

\subsection{A $2$-dimensional family of product-quotient surfaces}
In \cite[Theorem 4.15]{Pe11} the first author constructed a
$2$-dimensional family of product-quotient surfaces (having
precisely two ordinary double points as singularities) with
$p_g=q=2$, $K^2=6$ and whose Albanese map is a generically finite
quadruple cover. We will now recall the construction and show that
this family is actually contained in $\mathcal{M}_{\Phi}$.

Let us denote by $\mathfrak{A}_4$ the alternating group on four
symbols and by $V_4$ its Klein subgroup, namely
\begin{equation*}
V_4=\langle id, \, (12)(34), \, (13)(24), \, (14)(23) \rangle
\cong (\mathbb{Z}/2 \mathbb{Z})^2.
\end{equation*}
$V_4$ is normal in $\mathfrak{A}_4$ and  the quotient
$H:=\mathfrak{A}_4 / V_4$ is a cyclic group of order $3$. By using
Riemann's existence theorem it is possible to construct two smooth
curves $C_1$, $C_2$ of genus $4$ endowed with an action of
$\mathfrak{A}_4$ such that the only non-trivial stabilizers are the
elements of $V_4$. Then
\begin{itemize}
\item $E_i':=C_i/\mathfrak{A}_4$ is an elliptic curve;
\item the $\mathfrak{A}_4$-cover $f_i \colon C_i \to E_i'$ is
branched at exactly one point of $E_i'$, with branching order $2$.
\end{itemize}
It follows that the product-quotient surface
\begin{equation*}
\widehat{X}:=(C_1 \times C_2)/\mathfrak{A}_4,
\end{equation*}
where $\mathfrak{A}_4$ acts diagonally, has  two rational double
points of type $\frac{1}{2}(1, \, 1)$ and no other singularities. It
is straightforward to check that the desingularization $\widehat{S}$
of $\widehat{X}$ is a minimal surface of general type with
$p_g=q=2$, $K_{\wS}^2=6$ and that $\widehat{X}$ is the canonical
model of $\widehat{S}$.

The $\mathfrak{A}_4$-cover $f_i \colon C_i \to E_i'$ factors through
the bidouble cover $g_i \colon C_i \to E_i$, where $E_i:=C_i/V_4$.
Note that $E_i$ is again an elliptic curve, so there is an isogeny
$E_i \to E_i'$, which is a triple Galois cover with Galois group
$H$. Consequently, we have an isogeny
\begin{equation*}
p \colon E_1 \times E_2 \to \wA:=(E_1 \times E_2)/H,
\end{equation*}
where the group $H$ acts diagonally, and a commutative diagram
\begin{equation} \label{dia.example.pq}
\begin{xy}
\xymatrix{
C_1 \times C_2  \ar[d]_{g_1 \times g_2} \ar[rr]^{\tilde{p}} & & (C_1 \times C_2 )/ \mathfrak{A}_4 =\widehat{X} \ar[d]^{\widehat{\alpha}}  \\
E_1 \times E_2   \ar[rr]^{p} & & (E_1 \times E_2)/H = \wA.   \\
  }
\end{xy}
\end{equation}
In this way one constructs a two-dimensional family
$\mathcal{M}_{PQ}$ of product-quotient surfaces $\widehat{X}$ which
are canonical models  of surfaces of general type with $p_g=q=2$ and
$K^2=6$. The morphism $\hat{\alpha} \colon \widehat{X} \to
\widehat{A}$ is the Albanese map of $\widehat{X}$; it is a finite,
non-Galois quadruple cover.

\begin{proposition} \label{wT-in-Mphi}
The family $\mathcal{M}_{PQ}$ is contained in $\mathcal{M}_{\Phi}$.
\end{proposition}
\begin{proof}
We must show that $\widehat{X}$ belongs to $\mathcal{M}_{\Phi}$.
From the construction it follows that the Tschirnhausen bundle of
$\hat{\alpha} \colon \widehat{X} \to \widehat{A}$ is of the form
\begin{equation*}
\mathscr{E}^{\vee}= p_* \mathscr{D},
\end{equation*}
where $\mathscr{D}$ is a principal polarization of product type on
$E_1 \times E_2$. Since $K(\mathscr{D})=0$ one  clearly has $t^*_x
\mathscr{D} \neq \mathscr{D}$ for all $x \in \textrm{Ker}(p)
\setminus \{0\}$, so by \cite[Theorem 2, p. 248]{O71} the
canonical injection $\mathscr{O}_{\widehat{A}} \longrightarrow \EE
\otimes \EE^{\vee}$ induces isomorphisms
\begin{equation*}
h^i(\widehat{A}, \, \mathscr{O}_{\widehat{A}}) \longrightarrow
h^i(\widehat{A}, \, \EE \otimes \EE^{\vee}), \quad i=0, \, 1.
\end{equation*}
In particular $H^0(\widehat{A}, \, \EE \otimes \EE^{\vee} ) \cong
\mathbb{C}$, that is $\EE^{\vee}$ is simple. Then the claim
follows from Corollary \ref{cor.simple}.
\end{proof}

\begin{remark}
The branch locus of the Albanese map $\hat{\alpha} \colon
\widehat{X} \to \wA$ is a curve
$\widehat{B}=2(\widehat{B}_1+\widehat{B}_2)$, where $\widehat{B}_1$
and $\widehat{B}_2$ are the images via $p$ of two elliptic curves
belonging to the two natural fibrations of $E_1 \times E_2$; then
$\widehat{B}_1 \widehat{B}_2=3$. Note that $\widehat{B}$ is not of
the form described in Proposition \emph{\ref{prop.sing.covers}}; the
reason is that the $(1, \,3)$-polarized abelian surface $\wA=(E_1
\times E_2)/H$ is not general, since it is an \'{e}tale
$\mathbb{Z}/3 \mathbb{Z}$-quotient of a product of elliptic curves.
\end{remark}

\section{Open problems} \label{sec.open.prob}
\begin{itemize}
\item[$\boldsymbol{(1)}$] Is $\mathcal{M}_{\Phi}$ a \emph{connected}
component of $\mathcal{M}_{2, \, 2,\, 6}^{\textrm{can}}$? This is
equivalent to asking whether it is \emph{open} therein; in other words,
given a smooth family $\mathscr{X} \to \Delta$ over a small disk
such that $X_0 \in \mathcal{M}_{\Phi}$, is $X_t \in
\mathcal{M}_{\Phi}$ for $t$ small enough?

Note that any surface which is deformation equivalent to a surface
in $\mathcal{M}_{\Phi}$ must have Albanese map of degree $4$, since
the Albanese degree of a surface with $q=2$ and maximal Albanese
dimension is a topological invariant, see \cite{Ca91} and
\cite[Section 5]{Ca11}. This leads naturally to the next question:

\item[$\boldsymbol{(2)}$] What are the possible degrees for the
Albanese map of a minimal surface with $p_g=q=2$ and $K^2=6$?

And, more generally:

\item[$\boldsymbol{(3)}$] What are the irreducible/connected
components of $\mathcal{M}_{2, \, 2,\, 6}^{\textrm{can}}$?

\end{itemize}

\section*{Appendix 1. The divisors corresponding to $X$, $Y$, $Z \in H^0(A, \, \mathscr{L})$}

Let $D_X$, $D_Y$, $D_Z$ be the three divisors on $A$ which
correspond to the distinguished sections $X$, $Y$, $Z \in H^0(A, \,
\mathscr{L})$ defined by the  Schr\"odinger representation of the
Heisenberg group $\mathscr{H}_3$, see Section \ref{sec.constr}. In
this appendix we show that, for a general choice of the pair $(A, \,
\mathscr{L})$, the curves $D_X$, $D_Y$, $D_Z$ are smooth and
intersect transversally. We believe that this fact is well known to
the experts, however we give a proof for lack of a suitable
reference.

Let us start with a couple of auxiliary results.

\begin{lemma} \label{lem.triple}
Let $(B, \, \mathscr{M})$ be a principally polarized abelian surface
and $K_1 \cong \mathbb{Z}/3 \mathbb{Z}$ be a subgroup of the group
$B[3]$ of points of order $3$ on $B$. Then there exist a $(1, \,
3)$-polarized abelian surface $(A, \, \mathscr{L})$ and an isogeny
$f \colon A \to B$ of degree $3$ such that:
\begin{itemize}
\item[$\boldsymbol{(i)}$] $\mathscr{L}=f^* \mathscr{M}$;
\item[$\boldsymbol{(ii)}$] $K(\mathscr{L})= \ker f \oplus f^{-1}(K_1)$.
\end{itemize}
\end{lemma}
\begin{proof}
In a suitable basis, the period matrix for $B$ is
\begin{equation*}
\left(
  \begin{array}{cccc}
    z_{11} & z_{12} & 1 & 0 \\
    z_{21} & z_{22} & 0 & 1 \\
  \end{array}
\right),
\end{equation*}
where $Z:=\left(
\begin{array}{cc}
z_{11} & z_{12} \\
z_{21} & z_{22} \\
\end{array}
\right)$ satisfies ${}^tZ=Z$ and $\textrm{Im}\,Z
>0$. Then $B=\mathbb{C}^2/ \Lambda'$, with
\begin{equation*}
\Lambda'=\lambda_1 \mathbb{Z} \oplus \lambda_2 \mathbb{Z} \oplus
\mu_1 \mathbb{Z} \oplus \mu_2 \mathbb{Z}
\end{equation*}
and
\begin{equation*}
\lambda_1:=\left(
            \begin{array}{c}
              z_{11} \\
              z_{21} \\
            \end{array}
          \right), \quad
\lambda_2:=\left(
            \begin{array}{c}
              z_{12} \\
              z_{22} \\
            \end{array}
          \right), \quad
\mu_1:=\left(
            \begin{array}{c}
              1 \\
              0 \\
            \end{array}
          \right), \quad
\mu_2:=\left(
            \begin{array}{c}
              0 \\
              1 \\
            \end{array}
          \right).
\end{equation*}
Up to a translation, we may assume $K_1= \langle
\overline{\lambda_2/3} \rangle$, where the symbol
``$\overline{\phantom{1}}$" denotes the image in the complex torus.
The lattice
\begin{equation*}
\Lambda:=\lambda_1 \mathbb{Z} \oplus \lambda_2 \mathbb{Z} \oplus
\mu_1 \mathbb{Z} \oplus 3\mu_2 \mathbb{Z}
\end{equation*}
verifies $[\Lambda' : \Lambda]=3$, hence setting $A:=\mathbb{C}^2 /
\Lambda$ the identity map $\mathbb{C}^2 \to \mathbb{C}^2$ induces a
degree $3$ isogeny $f \colon A \to B$ such that $\ker f = \langle
\bar{\mu}_2 \rangle$.

Now the polarization $\mathscr{L}:=f^* \mathscr{M}$ is of type $(1,
\, 3)$ and satisfies
\begin{equation*}
K(\mathscr{L})= \langle \bar{\mu}_2, \, \overline{\lambda_2 /3}
\rangle=\ker f \oplus f^{-1}(K_1).
\end{equation*}
\end{proof}

\begin{lemma} \label{lem.translates}
Let $(A, \, \mathscr{L})$ and $(B, \, \mathscr{M})$ be as in Lemma
\emph{\ref{lem.triple}}. Then, up to a simultaneous translation, the
three divisors $D_X, \, D_Y, \, D_Z$ corresponding to $X$, $Y$, $Z
\in H^0(A, \, \mathscr{L})$ are given by
\begin{equation*}
f^*M, \quad f^*t_x^*M, \quad f^*t_{2x}^*M,
\end{equation*}
where $x \in B$ is a generator of $K_1$ and $M$ is the unique
effective divisor in the linear system $|\mathscr{M}|$.
\end{lemma}
\begin{proof}
Up to translation, we may assume that $\mathscr{L}$ is a line bundle
of characteristic zero with respect to the decomposition
$\Lambda=\Lambda_1 \oplus \Lambda_2$, where $\Lambda_1=\langle
\bar{\lambda}_1, \, \bar{\lambda}_2 \rangle$ and $\Lambda_2=\langle
\bar{\mu}_1, \, 3\bar{\mu}_2 \rangle.$ Hence, by the isogeny theorem
for finite theta functions (see \cite[page 168]{BL04}), it follows
$D_X=f^*M$.

On the other hand, if $y \in A$ is any point such that $f(y)=x$,
that is $y \in f^{-1}(K_1) \subset K(\mathscr{L})$, one has
\begin{equation*}
f^*t_x^*M \in |t_y^*f^* \mathscr{M}|=|t_y^* \mathscr{L}|=
|\mathscr{L}|
\end{equation*}
and the same holds for $f^*t_{2x}^*M$. In other words, $f^*t_x^*M$
and $f^*t_{2x}^*M$ are effective divisors corresponding to
sections of $H^0(A, \, \mathscr{L})$, and  straightforward
computations as in \cite[Chapter 6]{BL04} show that they are $D_Y$
and $D_Z$, respectively.
\end{proof}

Now we can prove the desired result.

\begin{proposition} \label{prop.X.Y.Z}
For a general choice of the pair $(A, \, \mathscr{L})$, the three
divisors $D_X$, $D_Y$, $D_Z \in |\mathscr{L}|$  are smooth and
intersect transversally.
\end{proposition}
\begin{proof}
It is sufficient to exhibit an example in which each claim is
satisfied.

First, take $B=J(C)$, the Jacobian of a smooth curve $C$ of genus
$2$, and let $\mathscr{M}$ be the natural principal polarization.
Choose as $K_1$ any cyclic subset of $B[3]$ and construct $(A, \,
\mathscr{L})$ as in Lemma \ref{lem.triple}. The unique effective
divisor $M \in |\mathscr{M}|$ is smooth, so the same is true for
its translates $t_x^*M$ and $t_{2x}^*M$. By Lemma
\ref{lem.translates} it follows that $D_X$, $D_Y$, $D_Z$ are
smooth, too. Hence they are smooth for a general choice of the
pair $(A, \, \mathscr{L})$.

Next, take $(B, \, \mathscr{M})=(E_1 \times E_2, \,
p_1^*\mathscr{M}_1 \otimes p_2^*\mathscr{M}_2)$, where $E_i$ is an
elliptic curve, $\mathscr{M}_i$ is a divisor of degree $1$ on
$E_i$ and $p_i \colon E_1 \times E_2 \to E_i$ is the projection.
Now choose a $3$-torsion point on $E_1 \times E_2$ of the form
$x=(x_1, \, x_2)$, where $x_i$ is a non-trivial $3$-torsion point
on $E_i$, take $K_1=\langle x \rangle$ and construct $(A, \,
\mathscr{L})$ as in Lemma \ref{lem.triple}. If $M$ is the unique
effective divisor in $|\mathscr{M}|$ then the three divisors $M$,
$t_x^*M$, $t_{2x}^*M$ are pairwise without common components and
intersect transversally, hence by Lemma \ref{lem.translates} the
same is true for $D_X$, $D_Y$, $D_Z$. Therefore $D_X$, $D_Y$,
$D_Z$ intersect transversally for a general choice of the pair
$(A, \, \mathscr{L})$.

This completes the proof.
\end{proof}

\begin{remark} \label{rem.D}
One can show that $(-1)_A \colon A \to A$ acts on $D_X$ as an
 involution,  whereas $D_Y$ and $D_Z$ are exchanged by $(-1)_A$. See
\emph{\cite[Exercise 13 p. 177]{BL04}}.
\end{remark}

%
\section*{Appendix 2. The MAGMA script used to calculate the equation of the branch divisor of $\alpha: S \to A$}
 {\scriptsize
\begin{verbatim}
QuadrupleCover:=function(Q,a,g)
//
// Given the field Q and the two parameters a and g of Proposition 2.4, the function QuadrupleCover
// gives:
// (1) the local equations of S "CoverAff5";
// (2) the local equation of the ramification divisor "RamCurve";
// (3) the local equation of the branch locus "AffBranch". Moreover
// (4) it checks whether these objects are singular or not, and it gives, if
// possible, a description of their singularities.
//
R<Y,Z,u,v,w>:=PolynomialRing(Q,5);
// Here Y,Z are sections in H^0(A,L)
A5 <Y,Z,u,v,w>:= AffineSpace(Q,5);
/////////////////////////////////////////////////////////////////////////////
// The local model of the cover S
/////////////////////////////////////////////////////////////////////////////
CoverAff5:=function(A5,R,a,g,T)
//
// Given a 5-dimensional affine space A5, a polynomial ring in the same
// coordinates R and three integers a,g,T (T is a translation), the function "CoverAff5"
// returns an affine scheme in A5 which is the local model of S with local equations in coordinates
// Y,Z,u,v,w.
//
b := -a; d := 0; if g eq 0 then e := 1; else e := -a^2/g; end if;
////////////////////////////////////////////////////////////////////////////
c12:=a*(Z-T); c13:=b*(Y-T); c14:=g*(Y-T); c15:=d; c16:=-g*(Z-T);
c23:=e; c24:=-b; c25:=-e*(Z-T); c26:=-d*(Y-T); c34:=-d*(Z-T);
c35:=e*(Y-T); c36:=-a; c45:=a*(Y-T); c46:=g; c56:=-b*(Z-T);
/////////////////////////////////////////////////////////////////////////////
a1:=1/2*c23; a2:=-c13; a3:=c12; a4:=1/4*c25-1/2*c34;
a5:=-1/2*c15-1/4*c23; a6:=c14; a7:=1/2*c26-1/4*c35; a8:=-c16;
a9:=1/2*c15-1/4*c23; a10:=c45; a11:=-1/2*c25; a12:=c24; a13:=c46;
a14:=-1/2*c26-1/4*c35; a15:=1/4*c25+1/2*c34; a16:=c56; a17:=-c36;
a18:=1/2*c35;
/////////////////////////////////////////////////////////////////////////////
b1:=-a1*a5+a2*a4-a2*a11-a3*a14+(a5)^2+a6*a8;
b2:=a2*a10+a3*a13-a4*a5-a6*a7; b3:=a2*a13+a3*a16-a4*a8-a7*a9;
b4:=-a1*a10+(a4)^2-a4*a11+a5*a10+a6*a13-a7*a12;
b5:=-a5*a13+a8*a10+a12*a17-a14*a15;
b6:=-a1*a16-a4*a17+(a7)^2-a7*a18+a8*a13+a9*a16;
/////////////////////////////////////////////////////////////////////////////
f1:=u^2-(a1*u+a2*v+a3*w+b1); f2:=u*v-(a4*u+a5*v+a6*w+b2);
f3:=u*w-(a7*u+a8*v+a9*w+b3); f4:=v^2-(a10*u+a11*v+a12*w+b4);
f5:=v*w-(a13*u+a14*v+a15*w+b5); f6:=w^2-(a16*u+a17*v+a18*w+b6);
/////////////////////////////////////////////////////////////////////////////
I:=ideal<R|[f1,f2,f3,f4,f5,f6]>;
/////////////////////////////////////////////////////////////////////////////
return Scheme(A5,Generators(I)); end function;
/////////////////////////////////////////////////////////////////////////////
// The ramification curve
/////////////////////////////////////////////////////////////////////////////
RamCurveA5:=function(Cover,A5,R)
//
// Given a local model of the "Cover", the affine space A5 in which the
// cover sits and a polynomial ring in the same coordinates,
// the function "RamCurveA5" returns an affine scheme in A5 which is
// the ramification curve of \ALPHA.
//
f1:= Generators(Ideal(Cover))[1]; f2:= Generators(Ideal(Cover))[2];
f3:= Generators(Ideal(Cover))[3]; f4:= Generators(Ideal(Cover))[4];
f5:= Generators(Ideal(Cover))[5]; f6:= Generators(Ideal(Cover))[6];
/////////////////////////////////////////////////////////////////////////////
I:=ideal<R|[f1,f2,f3,f4,f5,f6]>;
/////////////////////////////////////////////////////////////////////////////
L:=[Derivative(f1,u),Derivative(f1,v),Derivative(f1,w),
    Derivative(f2,u),Derivative(f2,v),Derivative(f2,w),
    Derivative(f3,u),Derivative(f3,v),Derivative(f3,w),
    Derivative(f4,u),Derivative(f4,v),Derivative(f4,w),
    Derivative(f5,u),Derivative(f5,v),Derivative(f5,w),
    Derivative(f6,u),Derivative(f6,v),Derivative(f6,w)];
/////////////////////////////////////////////////////////////////////////////
J:=Matrix(R,6,3,L); Ms:=Minors(J,3); J3:=ideal<R | Ms>;
RJ3:=Radical(J3); H:=RJ3+I; U:=Generators(H);
/////////////////////////////////////////////////////////////////////////////
return Scheme(A5,U); end function;
/////////////////////////////////////////////////////////////////////////////
// MAIN ROUTINE
/////////////////////////////////////////////////////////////////////////////
a:=a; g:=g; T:=0;
/////////////////////////////////////////////////////////////////////////////
printf "\n --------BEGINNING------- \n ";
printf "\n a=%o,\n g=%o,\n Q=%o \n", a,g,Q;
/////////////////////////////////////////////////////////////////////////////
CoverA5:=CoverAff5(A5,R,a,g,T); printf "\n The cover is a  ";
CoverA5; printf "\n The affine dimension of the cover is ";
Dimension(CoverA5); printf "The cover is singular: ";
IsSingular(CoverA5); if IsSingular(CoverA5) then
SingSchCoverA5:=SingularSubscheme(CoverA5); printf"The dimension of
the singular locus is: "; Dimension(SingSchCoverA5); if
Dimension(SingSchCoverA5) eq 0 then
SingPCoverA5:=SingularPoints(CoverA5); printf"The rational singular
points are: "; SingPCoverA5; else printf"The cover is not normal!";
end if; else RamCurveA5:=RamCurveA5(CoverA5,A5,R); printf "\n The
ramification divisor is a  "; RamCurveA5; printf "The affine
dimension of the ramification divisor is "; Dimension(RamCurveA5);
printf"The ramification divisor is singular: ";
IsSingular(RamCurveA5); printf"\n \n
---------------------------- \n \n";
/////////////////////////////////////////////////////////////////////////////
// The branch curve
/////////////////////////////////////////////////////////////////////////////
A2<q,l>:=AffineSpace(Q,2); f := map< A5 -> A2 | [Y,Z] >;
AffBranch:=f(RamCurveA5); printf"The branch divisor is singular:
"; IsSingular(AffBranch); if IsSingular(AffBranch) then
SingSchBranch:=SingularSubscheme(AffBranch); printf"The dimension
of its singular locus is "; Dimension(SingSchBranch); end if; if
Dimension(SingSchBranch) eq 0 then
SingPBranch:=SingularPoints(AffBranch); printf"The rational
singular points of the branch are "; SingPBranch; printf" The
branch is a "; AffBranch; printf" \n \n
---------------------------- \n \n";
/////////////////////////////////////////////////////////////////////////////
end if; end if; printf" \n ------------THE END---------- \n"; return
0; end function;
/////////////////////////////////////////////////////////////////////////////
\end{verbatim}
}
%

\bigskip
\bigskip

Matteo Penegini, Dipartimento di Matematica ``Federigo Enriques'', Universit\`{a} degli Studi di Milano, Via Saldini 50, 20133 Milano, Italy. \\
\emph{E-mail address:}
 \verb|matteo.penegini@unimi.it| \\ \\

Francesco Polizzi, Dipartimento di Matematica e Informatica,
Universit\`{a} della
Calabria, Cubo 30B, 87036 Arcavacata di Rende (Cosenza), Italy.\\
\emph{E-mail address:} \verb|polizzi@mat.unical.it|

\end{document}